\documentclass[12pt,a4paper,twoside, reqno]{amsart}
\usepackage{amsfonts, amsthm, amsmath, amssymb}
\usepackage{mathrsfs,amsmath}
\usepackage{hyperref}
\usepackage{blindtext}
\hypersetup{colorlinks=false}

\usepackage[margin=2.6cm]{geometry}

\usepackage{helvet}

\newcommand{\shortmod}{\ensuremath{\negthickspace \negthickspace \negthickspace \pmod}}

\newcommand{\sumstar}{\sideset{}{^*}\sum}

\newcommand{\e}[2]{e\left(\frac{#1}{#2}\right)}

\usepackage{amsthm}
\newtheorem{theorem}{Theorem}
\newtheorem{lemma}{Lemma}

\begin{document}
	
	\author{ Sumit Kumar,  K. Mallesham, Suraj Panigrahy}
	
	\title{ Second moment of $\textrm{GL(3)} \times \textrm{GL(2)}$ $L$--functions}
	\address{ Suraj Panigrahy \newline { Department of Mathematics and Statistics, Indian Institute of Technology, Bombay, India; \newline  
    Email: suraj.grahy12@gmail.com  
	} }	
	
	\address{Sumit Kumar \newline {  Department of Mathematics and Statistics, Indian Institute of Technology, Bombay, India;\newline 
			Email: sumit@math.iitb.ac.in
	}}
	
	\address{ K. Mallesham  \newline { Department of Mathematics and Statistics, Indian Institute of Technology, Bombay, India; \newline  
    Email: iitm.mallesham@gmail.com
	} }
	
	\subjclass[2010]{Primary 11F66, 11M41; Secondary 11F55.}
	\date{\today}
	\keywords{Maass forms, subconvexity, Rankin-Selberg $L$-functions.}

\maketitle
    
	\begin{abstract}
    For $M_1$ and  $ M_2$  two distinct primes, let   $ H_k^\star(M_1M_2, \psi)$ denote the set of primitive newforms of level $M_1M_2$, weight $k\geq 3$ and  Nebentypus $\psi$ of conductor $M_1$.  Let  $\pi$ be a fixed $SL(3, \mathbb{Z})$ Hecke cusp form. We prove a Lindelöf--consistent upper bound for the second moment  
        \[ \mathop{ \sum_{\substack{\psi(M_1) \\ \psi(-1)=(-1)^k }}} \sideset{}{^h}\sum_{f \in H_k^{\star}(M_1M_2,\psi)} |L(1/2, \pi \times f)|^2 \ll_{\pi,\epsilon} M_1^{1+\epsilon}\] 
        in the range $M_2\leq M_1^{1+\epsilon}$.
	\end{abstract}

\section{Introduction}
Study of moments of  $L$-functions is a recurring theme  in modern  number theory.  Given a family  of $L$-functions $\{L(s, F)  \}$ associated to an  automorphic form $F \in  \mathcal{F}$, we define the $2k$-th moment at the central point $s=1/2$ 
\[M_{2k}(\mathcal{F})= \sum_{F \in \mathcal{F}}|L(1/2, F)|^{2k}.\]
These moments
are  the  natural objects to study as they illuminate
structure of the family and display beautiful symmetries in the family. A broad goal  is to get an asymptotic formula of the form
\[M_{2k}(\mathcal{F})= \mathrm{Main \ Term}+ o(\mathrm{Main \ Term}),\]
as $|\mathcal{F}| \rightarrow \infty$, where $\mathrm{Main \ Term} \asymp |\mathcal{F}|$.
Such estimates often yield  interesting arithmetic applications, for example, non-vanishing of $L$-values, subconvexity of $L$-functions etc. 
Unfortunately, there are only a handful of known asymptotic formulae in the literature. 

For GL(1) family
\begin{align}\label{zeta}
   M_{2k}(T,q):= \sum_{\chi^{\star}(q)}\int_{-T}^{T} |L(1/2+it, \chi)|^{2k} \, \mathrm{d}t,
\end{align}
asymptotic formulae are known due to Hardy-Littlewood (see \cite{T}) for $k=1,q=1$ and due to Ingham (see \cite{T})   for $k=2,q=1$, and  due to  Heath-Brown \cite{HB81} for $k=1, T=0$ and  Young \cite{Y07} for $k=2,T=0$.  There is also  some recent progress for $k=3$ and $4$ for a large family (also including the sum over the moduli $q$) by Chandee et. al. in  \cite{CLMR23}, \cite{CLMR24}.

This problem gets challenging for large $k$ and for `higher rank' family. Even a  slightly weaker  problem  of  getting  the  Lindel\"of--consistent upper bound $$M_{2k}(\mathcal{F}) \ll_\epsilon |\mathcal{F}|^{1+\epsilon},  \ |\mathcal{F}| \rightarrow \infty $$ is mostly out of reach.   
Let $\mathcal{F}= \mathcal{F}_2(q, T)  $  denote the set of    Maass cusp  forms  on  GL(2)  of  level $q$   and Laplace eigenvalue $1/4+t_f^2, $  $  T\leq t_f\leq 2T$. For this family
some  notable known results  are  due to Kowalski-Michel-Vanderkam \cite{KMV00} in the level aspect and due to Jutila-Motohashi \cite{JM06}  in the spectral aspect    for the fourth moment.
The sixth moment was first studied by Young \cite{MY09} who proved 
\begin{align}\label{6 moment}
    \sum_{f \in \mathcal{F}_2(1,T)} |L(1/2+it_f,f)|^6 \ll_\epsilon T^{2+\epsilon},
\end{align}
In the same paper  he also estimated  the second moment 
\begin{align*}
    \sum_{f \in \mathcal{F}_2(1,T)} |L(1/2+it_f,\pi \times f)|^2\ll_\epsilon T^{2+\epsilon},
\end{align*}
of  Rankin--Selberg $L$-function associated to (GL(3), GL(2)) pair of automorphic forms $(\pi,f)$. This may be thought of as the cuspidal analog of  \eqref{6 moment}. In a companion paper \cite{MY11}, Young proved 
\[\int_{-T^{1-\epsilon}}^{T^{1-\epsilon}}\sum_{f \in \mathcal{F}_2(1,T)} |L(1/2+it,\pi \times f)|^2\ll_\epsilon T^{3+\epsilon}.\]
Motivated by \cite{MY11}, Djankovi\'c \cite{D11} proved a `non-archimedian' analog of \eqref{6 moment}: for $M$ prime and $k\geq 3$ odd integer 
\[  \mathop{ \sum_{\substack{\psi(M) \\ \psi(-1)=(-1)^k }}} \sideset{}{^h}\sum_{f \in H_k^{\star}(M,\psi)} |L(1/2,  f)|^6 \ll M^{1+\epsilon},\]
where  for any integer $M$, $ H_k^\star(M, \psi)$ denote the set
of primitive (normalized such that the first coefficient is  $1$) newforms   which forms an orthogonal basis of $S_k^{\mathrm{new}}(M, \psi)$.   A key ingredient of  \cite{D11} is the asymptotic large sieve of Iwaniec-Li \cite{IL07} followed by Voronoi summation formula.  

Recently Qi \cite{Qi25} extended Young's result \cite{MY09} to the short-interval range $ T < t_f \leq T+\sqrt{T}$ and obtained Lindel\"of consistent bound in the coveted family range ($ |\mathrm{family}| \asymp \mathrm{|conductor|}^{1/4}$). Purpose of this paper to prove a `non-archimedian' analog of Qi \cite{Qi25}. 

Let $M_1$ and $M_2$ be two distinct primes such that  $M_1 <M_2$. 
We consider for $M=M_1M_2$ the second moment
\[ \mathcal{S}_\pi(M):= \mathop{ \sum_{\substack{\psi(M_1) \\ \psi(-1)=(-1)^k }}} \sideset{}{^h}\sum_{f \in H_k^{\star}(M,\psi)} |L(1/2, \pi \times f)|^2\]
\begin{theorem}\label{main thm}
   Let $M^\epsilon \leq M_2 \leq M_1^{1+\epsilon}$.  
   Then  we have 
   \[ \mathcal{S}_\pi(M) \ll_\pi M_1M^\epsilon .\]
\end{theorem}

 The above theorem yields a Lindel\"of consistent bound in the range $ M_2 \asymp M_1$.  Our result extend \cite{D11}  in  the shorter range $ |\mathrm{family}| \asymp M_1^2M_2$. The case $M_2 \geq M_1^{1+\epsilon}$ turns out to be out of reach at the moment as the size of family gets shorter which makes this super hard to analyze from analytic point of view. A key ingredient in the proof is the large sieve inequality which we prove in Lemma \ref{hybrid large-sieve with level}.

\section{Preliminaries}
\subsection{Petersson formula}
For $k$ and $q$ two integers, $k \geq 2,$ and $\psi$ a Dirichlet character of modulus $q$ and conductor $\hat{q},$ let $S_k(q, \psi)$ denote the complex vector space of weight $k $ and holomorphic cusp forms with level $q$ and nebentypus $\psi$. Let $S_k^{\mathrm{new}}(q, \psi)$ be the space of newforms in  $S_k(q, \psi).$  We let $ H_k^\star(q, \psi)$ denote the set
of primitive (normalized such that the first coefficient is  $1$) newforms    which forms an orthogonal basis of $S_k^{\mathrm{new}}(q, \psi)$. We  extend it to $H_k(q, \psi)$--an orthogonal Hecke--eigenbasis of $S_k(q, \psi)$. 
We denote
\[ \sideset{}{^h}\sum_{f \in H_k^{}(q,\psi)} \alpha_f:=  \sideset{}{}\sum_{f \in H_k^{}(q,\psi)} w_f^{-1}\alpha_f, \ \   w_f^{-1} = \frac{\Gamma(k-1)}{(4\pi)^{k-1} \| f\|^2}. \]
 We state the  Petersson formula as follows:
\begin{align}\label{petersson}
    \sideset{}{^h}{\sum}_{f \in H_k(q, \psi)}
 \lambda_f(n)\overline{\lambda_f(m)} 
=\delta(n-m) + 2 \pi i^{-k}
\sum_{c=1}^\infty  \frac{S_\psi(m,n;qc)}{cq} J_{k-1}\left( \frac{4\pi \sqrt{nm}}{pqc}\right),
\end{align}
where  $\lambda_f(.)$ are the Hecke normalized  eigenvalues  of  $f$.

\subsection{Bessel function}
Let $k \geq 2$ and $x>0$. Then the $J$-Bessel function splits as
\[J_k(2\pi x)=W_k(x)e(x)+\overline{W_k}(x)e(-x),\]
where $W_k(x)$  is a smooth function defined on $(0, \infty)$, and it satisfies the following bound:
\[x^jW_k^{(j)}(x)\ll_{k,j}  \{x^{k-1}, \frac{x}{(1+x)^{3/2} }\}, \ j \geq 0.  \]

\subsection{Duality principle and large sieve inequality}	
\begin{lemma}[Duality principle]\label{dualitylemma}
		Let $\phi: \mathbb{Z}^2\rightarrow \mathbb{C}$. For any complex numbers $a_m$, \begin{align*}
			\sum_n\left|\sum_m a_m\phi(m,n)\right|^2\ll \left( \sum_{ m} |a_m|^2\right)\sup_{\|\beta\|_2=1}\sum_m\left|\sum_n\beta(n)\phi(m,n)\right|^2,
		\end{align*}
		where the supremum is taken over all sequences of complex numbers $\beta(n)$ such that $$\|\beta\|_2=\sqrt{\sum_n|\beta(n)|^2}=1.$$
	\end{lemma}

    
\begin{lemma}\label{C=1}
    Let $a_m$ be any complex numbers and $D>0$ be a fixed  integer.  Then  we have 
       \begin{align}
   S(N;D):= \sum_{\gamma (D)}  \Big|\sum_{N \leq m < 2N} a_m  e\left( \frac{\gamma m}{D} \right) \Big| ^2  \leq (N +  D) \sum_m|a_m|^2.
\end{align}
\end{lemma}
\begin{proof}
 We open the absolute valued square to arrive at 
\begin{align*}
   S(N;D) &\leq  \sum_{m_1}  \sum_{m_2} a_{m_1}\overline{a}_{m_2} \sum_{\gamma (D)}  e\left(\frac{ (m_1-m_2) \gamma }{D} \right) \\
   &\leq  D  \mathop{\sum \sum }_{m_1=m_2 (D)}   |a_{m_1}\overline{a}_{m_2}| \\
  &\leq (D+N) \sum_m |a_{m}|^2.
\end{align*} 
\end{proof}
\begin{lemma}\label{large sieve with level}
    Let $a_m$ be any complex numbers and $D>0$ be a fixed  integer. Let $1 \leq M\leq N$. Then  we have 
    \begin{align}\label{small N}
   S(C,N,M;D):= \mathop{\sum_{\substack{c \sim C \\ (c,D)=1}}}\sumstar_{\gamma (cD)}  \Big|\sum_{N \leq m < N+M} a_m  e\left( \frac{\gamma m}{cD} \right) \Big| ^2  \ll  C^\epsilon(M +  DC^2) \sum_m|a_m|^2.
\end{align}
\end{lemma}
\begin{proof}
By a change of variable, we see that 
\begin{align*}
    S(C,N,M;D)= \mathop{\sum_{\substack{c \sim C \\ (c,D)=1}}}\sumstar_{\gamma (cD)}  \Big|\sum_{1 \leq m < M+1} a({N-1+m})  e\left( \frac{\gamma m}{cD} \right) \Big| ^2.
\end{align*}

We  apply the duality principle (see Lemma \ref{dualitylemma}) to analyse $S(C,N,M;D)$. Indeed, we have 
     \begin{align*}
      S(C,N,M;D) \leq 
             \sum_m |a_{m}|^2\sup_{\|\beta\|_2=1} \sum_{1 \leq m < M+1} \Big|    \mathop{\sum_{ \substack{c \sim {C} \\ (c,D)=1 }}} \ \sumstar_{\substack{ \gamma \shortmod{cD } }  }  \beta(\gamma,d)  e\left(\frac{ m \gamma }{cD} \right)   \Big|^2.
\end{align*}
   Next we extend the sum over $m$ to $ -2MQ \leq m \leq 2MQ $, where $Q>2$ is  a parameter  such that $$Q=C^{5\epsilon}+\frac{C^{2+\epsilon}D}{M}. $$ 
   To this end, we introduce 
 a smooth bump function $w(x)$ such that $0\leq w(x)\leq 1$, $w(x)=1$ for $x \in [-1,1] $ and $w(x)=0$, for $x \notin (-2,2)$, and   $x^j w^{(j)}(x) \ll_{j} 1$. Thus we have 
 \begin{align*}
     S(c,N;D)
       &\leq   \sum_m |a_{m}|^2   \sup_{\|\beta\|_2=1} \sum_{m}w\left( \frac{m}{MQ} \right) \Big|  \mathop{\sum_{ \substack{c \sim {C} \\ (c,D)=1 }}} \ \sumstar_{\substack{ \gamma \shortmod{cD } }  }   \beta(\gamma,c)  e\left(\frac{ m \gamma }{cD} \right)   \Big|^2.
\end{align*}
Opening the absolute valued square, we arrive at
\begin{align*}
      \mathop{\sumstar_{ \substack{c, c^\prime \sim  {C} }}} \ \sumstar_{\substack{ \gamma \shortmod{cD}}  } \,\sumstar_{\substack{ \gamma^\prime \shortmod{c^\prime D}}  }   \beta(\cdots) \overline{\beta}(\cdots)  
      \sum_m w(m/MQ) e\left(\frac{ m(\gamma c^\prime-\gamma^\prime c ) }{Dcc^\prime} \right).
\end{align*}
  We proceed to apply the  Poisson summation formula to the $m$-sum. The dual $m$-sum is given by 
   \begin{align}\label{dual m 2-}
    \sum_m \cdots= {MQ}{ }\sum_{m  }  \delta( \gamma c^\prime-\gamma^\prime c \equiv -m\, \mathrm{mod}\, Dcc^\prime)  I_1(\cdots),
\end{align}
where 
\begin{align}\label{integral poisson 3}
    I_1(\cdots)= \int w(z)    e\left( \frac{-mMQz}{Dcc^\prime}\right) \mathrm{d}z.
\end{align}
An integration by parts argument shows that $ I_1(\cdots)$ is negligibly small ($\ll D^{-2025}$)unless 
$$ |m| \leq N_2:=C^{\epsilon} \frac{DC^2}{ MQ}. $$
Note that $N_2<1$ as 
 $DC^2C^\epsilon < MQ=MC^\epsilon+ C^{2+\epsilon}D.$   Thus only the zero--frequency survives, i.e., we have $m=0$.
Using the congruence, we conclude that $c=c^\prime$ and $\gamma=\gamma^\prime (Dc)$. Hence 
 \begin{align*}
    S(C,N;D)  \ll \sum_m|a_m|^2MQ &     \mathop{\sumstar_{ \substack{c \sim {C}{  }
       } } }  \ \sumstar_{\substack{ \gamma \shortmod{cD}}  }   |\beta(\gamma,c) \overline{\beta}( \gamma,  c)|   \ll MQ \sum_m|a_m|^2.
\end{align*}
Hence the lemma follows.
\end{proof}
We also need a hybrid-type large sieve bound.
\begin{lemma}\label{hybrid large-sieve with level}
    Let $a_m$ be  any complex numbers  and $D>0$ be a fixed  integer and $T >0 $ be a real number. Let $f$ be a $C^1$ function on $[N,2N]$ such that $f^\prime$ does not vanish. Let  $X=\sup_{x \in [N,2N]} \frac{1}{|f^\prime(x)|}$. Then we have  
    \begin{align*}
   S_T(C,N;D):=\mathop{\sum_{\substack{c \sim C \\ (c,D)=1}}}\sumstar_{\gamma (cD)} \int_{-T}^{T} \Big|\sum_{N \leq m <  2N} a_m  e(tf(m)) e\left( \frac{\gamma m}{cD} \right) \Big| ^2\mathrm{d}t  \ll  (C^2DT+X) \sum_m|a_m|^2.
\end{align*}
\end{lemma}
\begin{proof}
Let $W$ be a non-negative smooth bump  function supported on $[1/2,3/2]$ such that $W=1$ on $[1,2]$. Thus we have 
  \begin{align*}
   S_T(C,N;D) \leq  \mathop{\sum_{\substack{c \sim C \\ (c,D)=1}}}\sumstar_{\gamma (cD)} \int_{0}^{\infty}W(t/T) \Big|\sum_{N \leq m <  2N} a_m  e(tf(m)) e\left( \frac{\gamma m}{cD} \right) \Big| ^2\mathrm{d}t.
\end{align*}
For any sequence of complex numbers $(b_m)$, we have 
\begin{align}\label{open the abs}
    \int_{T}^{2T } \Big|\sum_{N \leq m <  2N} b_m  e(tf(m))  \Big| ^2\mathrm{d}t \leq  \sum_{m_1}  \sum_{m_2}b_{m_1} \overline{b}_{m_2} \int_{0}^{\infty} W\left( \frac{t}{T} \right) e\left( t(f(m_1)-f(m_2) \right)  \mathrm{d}t 
\end{align}
 By a change of variable, we see that the $t$-integral is negligibly small unless $$ |f(m_1)-f(m_2)| \leq C^{\epsilon}/T. $$ 
 By the mean-value theorem, we have
 \[|m_2-m_1| \leq {|f(m_1)-f(m_2)|X} \leq {C^\epsilon X}/{T}.\]
Let $Y=\min\{C^\epsilon X/T,N \}$. We dissect the sum over $m_1$ and $m_2$ into the intervals $$I_k=[N+kY,N+(k+1)Y), \ 0 \leq k \ll N/Y $$   as follows:
\[    \sum_{N \leq m_1<2N}  \sum_{N \leq m_2 <2N}\cdots = \sum_{0\leq k\ll N/Y}\sum_{m_1 \in I_k} \sum_{m_2 \in J_k} \cdots,\]
where $J_k=I_k+([-Y,Y]\cap \mathbb{Z})\cap [N,2N).$
Thus the right side of \eqref{open the abs} is 
\begin{align*}
   &\leq \frac{1}{2} \int_{0}^{\infty} W\left( \frac{t}{T} \right) \sum_k \Big| \sum_{m_1 \in I_k} b_{m_1} e\left( t(f(m_1) \right) \Big|^2+\Big|\sum_{m_2 \in J_k} \bar{b}_{m_2} e\left( -tf(m_2) \right) \Big|^2\mathrm{d}t.
     \end{align*}
We focus on the second term, as the first term is similar, which, on taking  $b_m=a_m e\left(\frac{\gamma m}{cD}\right)$ is given by 
\begin{align*}
    \frac{1}{2} \int_{0}^{\infty} W\left( \frac{t}{T} \right) \sum_k \Big|\sum_{m \in J_k} a_m e\left(\frac{\gamma m}{cD}\right) e\left( tf(m) \right) \Big|^2\mathrm{d}t.
\end{align*}
Thus on summing over $c$ and $\gamma$, we arrive at 
\begin{align*}
     \frac{1}{2} \int_{0}^{\infty} \mathop{\sum_{\substack{c \sim C \\ (c,D)=1}}}\sumstar_{\gamma (cD)} W\left( \frac{t}{T} \right) \sum_k \Big|\sum_{m \in J_k} a_m e\left(\frac{\gamma m}{cD}\right) e\left( tf(m) \right) \Big|^2\mathrm{d}t.
\end{align*}
On applying Lemma \ref{large sieve with level} with the coefficients $a_m e(t(f(m)))$, the above is  
\[ \ll C^\epsilon \sum_{k \ll N/Y} T (C^2D+ Y)\sum_{m \in J_k}|a_m|^2 \ll C^\epsilon (TC^2D + X )\sum_m|a_m|^2.\]
The lemma follows.
\end{proof}

We  borrow the following  lemma from  \cite{MY11}.
\begin{lemma} \label{matt lemma}
For all integers $m$, $n$, and positive integers $c$, we have
 \begin{equation}\label{eq:Stwist}
  S(m,n;c) \e{-m-n}{c} = \sum_{ab =c} \sum_{\substack{x \shortmod{b} \\ (x(x+a),b)=1}} \e{\overline{x} m - (\overline{x+a})n}{b}.
 \end{equation}
\end{lemma}

\subsection{Automorphic  forms on $\mathrm{GL(3)}$}
Let $\pi$ be a Hecke--Maass cusp form of type $(\nu_{1}, \nu_{2})$ for $\mathrm{SL(3, \mathbb{Z})}$. Let $\lambda_{\pi}(n,r)$ denote the normalised Fourier coefficients of $\pi$. Let 
\[ {\alpha}_{1} = - \nu_{1} - 2 \nu_{2}+1, \, {\alpha}_{2} = - \nu_{1}+ \nu_{2}  \ \mathrm{and} \ {\alpha}_{3} = 2 \nu_{1}+ \nu_{2}-1\]
be the spectral parameters for $\pi$.
Let $g$ be a compactly supported smooth function on  $ (0, \infty )$ and 
\[ \tilde{g}(s) = \int_{0}^{\infty} g(x) x^{s-1} \mathrm{d}x\]
 be its Mellin transform. For $\ell= 0$ and $1$, we define
\begin{equation}
	\gamma_{\ell}(s) :=  \frac{\pi^{-3s-\frac{3}{2}}}{2} \, \prod_{i=1}^{3} \frac{\Gamma\left(\frac{1+s+{\alpha}_{i}+ \ell}{2}\right)}{\Gamma\left(\frac{-s-{\alpha}_{i}+ \ell}{2}\right)}.
\end{equation}
Set $\gamma_{\pm}(s) = \gamma_{0}(s) \mp \gamma_{1}(s)$ and let 
\begin{align}\label{gl3 integral transform}
	G_{\pm}(y) = \frac{1}{2 \pi i} \int_{(\sigma)} y^{-s} \, \gamma_{\pm}(s) \, \tilde{g}(-s) \, \mathrm{d}s,
\end{align}
where $\sigma > -1 + \max \{-\Re({\alpha}_{1}), -\Re({ \alpha}_{2}), -\Re({\alpha}_{3})\}$. 
\begin{lemma} \label{gl3voronoi}
	Let $g(x)$ and  $\lambda_{\pi}(n,r)$ be as above. Let $a, \, q \in \mathbb{Z}$ with $q\geq 1, (a,q)=1,$ and  $a\bar{a} \equiv 1(\mathrm{mod} \ q)$. Then we have
	\begin{align*} \label{GL3-Voro}
		\sum_{n=1}^{\infty} \lambda_{\pi}(r,n) e\left(\frac{an}{q}\right) g(n) 
		=q  \sum_{\pm} \sum_{n_{1}|qr} \sum_{n_{2}=1}^{\infty}  \frac{\lambda_{\pi}(n_1,n_2)}{n_{1} n_{2}} S\left(r \bar{a}, \pm n_{2}; qr/n_{1}\right) G_{\pm} \left(\frac{n_{1}^2 n_{2}}{q^3 r}\right),
	\end{align*} 
	where  $S(a,b;q)$ is the  Kloosterman sum which is defined as follows:
	\[ S(a,b;q) = \sideset{}{^\star}{\sum}_{x \,{\mathrm{mod}}\,  q} e\left(\frac{ax+b\bar{x}}{q}\right).\]
\end{lemma}
\begin{proof}
	See \cite{Li09}. 
\end{proof}
In the following lemma we extract  oscillations of  $G_{\pm}$. 
\begin{lemma} \label{GL3oscilation}
	Let $G_{\pm}(x)$ be as above,  and  $g(x) \in C_c^{\infty}(X,2X)$. Then for any fixed integer $K \geq 1$ and $xX \gg 1$, we have
	\begin{equation*}
		G_{\pm}(x)=  x \int_{0}^{\infty} g(y) \sum_{j=1}^{K} \frac{c_{j}({\pm}) e\left(3 (xy)^{1/3} \right) + d_{j}({\pm}) e\left(-3 (xy)^{1/3} \right)}{\left( xy\right)^{j/3}} \, \mathrm{d} y + O \left((xX)^{\frac{-K+5}{3}}\right),
	\end{equation*}
	where $c_{j}(\pm)$ and $d_{j}(\pm)$ are some   absolute constants depending on $\alpha_{i}$,  $i=1,\, 2,\, 3$.  
\end{lemma}
\begin{proof}
	See   \cite[Lemma 6.1]{Li09}.
\end{proof}

\section{Petersson trace formula}
As a first step, we apply approximate functional equation  (see Th. 5.3 and Prop. 5.4 in \cite{IK04}, Theorem 3.9 in \cite{BKL18}) we get  
\begin{align*}
    |L(1/2, \pi \times f)|^2 \ll  \Big| \sum_{n=1}^{\infty} \frac{ \lambda_{\pi \times f}(n)}{\sqrt{n}}W\left( \frac{n}{M^{3/2}}\right)\Big|^2,
\end{align*}
where 
\[ \lambda_{\pi \times f}(n)= \sum_{\ell^2m=n} \lambda_\pi(\ell,m)\lambda_f(m)\psi(\ell) \]

and $W$ is a smooth function satisfying
\[x^jV^{(j)}(x) \ll \frac{1}{(1+|x|)^A}, \ \ \mathrm{for \ any \ } A>0. \]
Using a dyadic partition of unity, we further see that for any $A> 1$
\begin{align}\label{AFE}
    |L(1/2, \pi \times f)|^2 \ll \sup_\ell \sup_{N\ell^2 \ll M^{3/2+\epsilon}}\sum_{\ell} \frac{|S_\ell(N)|^2}{N}+M^{-A},
    \end{align}

where 
\[S_\ell(N)=\sum_{m=1}^\infty \lambda_\pi(\ell,m) \lambda_f(m)V_{\ell,N}\left(\frac{m}{N}\right)\]
where $V_{\ell, N}(x)$ is a smooth function supported in  $[1,2] $ and satisfying 
\begin{align}\label{bd for V}
     V_{\ell,N}^{(j)}(x) \ll_j 1, \ \ j \geq 0.
\end{align}
In further analysis we only require an upper bound for $V_{\ell,N}$  of the form  \eqref{bd for V} which is independent of $\ell$ and $N$. Henceforth we  ignore the dependence
on $\ell$ and $N$ and assume $V_{\ell,N}=V$.  Thus on applying \eqref{AFE} to $\mathcal{S}_\pi(M)$, we get 
\begin{align}
    \mathcal{S}_\pi(M) \ll \sup_\ell \sup_{N\ell^2 \ll M^{3/2+\epsilon}} \sum_{\psi(M_1)}\,  \sideset{}{^h} \sum_{f \in H_k^{\star}(M, \psi)}\frac{|S_\ell(N)|^2}{N}+M^{-A}
\end{align}
Next we extend $H_k^{\star}(M, \psi)$ to $H_k(M, \psi)$  by positivity and proceed to apply Petersson trace formula. Indeed, we have 
\begin{align}\label{after peter}
    \mathcal{S}_\pi(M) \ll \sup_\ell \sup_{N\ell^2 \ll M^{3/2+\epsilon}} \frac{1}{N}  \mathop{\sum\sum}_{m,m^\prime=1}^\infty \lambda_\pi (\ell, m) \overline{\lambda_{\pi}}(\ell,m^\prime) Z_{M}(m,m^\prime)V\left(\frac{m}{N}\right)\overline{V}\left(\frac{m^\prime}{N}\right) +M^{-A},
\end{align}
where 
\[Z_M(m,m^\prime)=  \sum_{\psi(M_1)}\sum_{f \in H_k^{}(M, \psi)}\omega_f^{-1}  \lambda_f(m) \overline{\lambda_f(m^\prime)}.\]
On applying \eqref{petersson} to the above expression we get that 
\[Z_M(m,m^\prime)= \frac{M_1-1}{2} \delta(m-m^\prime)+\mathcal{O}\mathcal{D}(m,m^\prime),\]
where
\begin{align*}
    \mathcal{O}\mathcal{D}(m,m^\prime)= 2 \pi i^{-k} \mathop{\sum}_{ \substack{\psi(M_1) \\ \psi(-1)=(-1)^k}}
\sum_{c=1}^\infty  \frac{S_\psi(m,m^\prime;Mc)}{Mc} J_{k-1}\left( \frac{4\pi \sqrt{mm^\prime}}{Mc}\right).
\end{align*}
Next we use the expansion for $x>0$
\[J_k(2\pi x)=W_k(x)e(x)+\overline{W_k}(x)e(-x),  \] 
\[x^jW_k^{(j)}(x)\ll_{k,j}  \min\{x^{k-1}, \frac{x}{(1+x)^{3/2} }\}, \ j \geq 0. \]
Let $x<M^{-2025}.$ Then  $W_k(x) \ll  M^{-2025}$. Thus $ \mathcal{O}\mathcal{D}(m,m^\prime)$ is very small. In the case, $ x\geq M^{-2025},$ 
 we apply a dyadic partition of unity to the sum over $c$ to write
\begin{align*}
    &\mathcal{O}\mathcal{D}(m, m^\prime)=\sum_{C \ll \frac{NM^{2025}}{M}} {\mathcal{O}\mathcal{D}}_C(m, m^\prime),
\end{align*}

where
\begin{align}\label{OD}
   {\mathcal{O}\mathcal{D}}_C(m, m^\prime)=\frac{2 \pi i^{-k}}{M} \sum_{\beta \in \{\pm\}}
\mathop{\sum}_{\substack{c \sim C \\  }} \frac{\mathcal{C}_{Mc}(m, m^\prime)}{c} W^{\beta}\left(\frac{2 \sqrt{mm^\prime}}{Mc}\right),
\end{align}
and
$W^{\beta}(x)=e(\beta x)W_{k-1}^{\beta}(x)$ with
$W_{k-1}^{+}=W_{k-1}$, $W_{k-1}^{-}=\overline{W}_{k-1}$ and 
\[\mathcal{C}_{Mc}(m,m^\prime)=\mathop{\sum}_{ \substack{\psi(M_1) \\ \psi(-1)=(-1)^k}}
  {S_\psi(m,m^\prime;Mc)}.\]
 On combining  the above analysis together, we infer that 
\begin{align}\label{S M after peter}
    \mathcal{S}_\pi(M) 
     \ll \frac{M_1-1}{2}+\sup_\ell \sup_{N\ell^2 \ll M^{3/2+\epsilon}} \sum_{C   }\frac{|\mathrm{OD}_{\pi}(\ell,C, N)| }{N},
\end{align}
where 
    \begin{align}\label{OD coprime}
         \mathrm{OD}_{\pi}(\ell,C, N)&= \mathop{\sum\sum}_{m,m^\prime=1}^\infty \lambda_\pi (\ell, m) \overline{\lambda_{\pi}}(\ell,m^\prime) {\mathcal{O} \mathcal{D}}_C(m,m^\prime) V\left(\frac{m}{N}\right)\overline{V}\left(\frac{m^\prime}{N}\right),
         \end{align}
\section{Separation of variables}

\subsection{Character sum}
In this section we separate $m$ and $m^\prime$ in the  character sum $  \mathcal{C}_{Mc}(m,m^\prime)$  using Lemma \ref{matt lemma}. 
\begin{lemma} \label{separ of char}
Let $c=c_0 M_1^{b_1} M_2^{b_2}$, $b_1 \geq 0$ and $ b_2 \geq 0$. Then we have  $0\leq b_1, b_2 \ll \log C$ and  
\begin{align}
        \mathcal{C}_{Mc}(m,m^\prime)= \frac{M_1-1}{2} \sum_{\eta \in \{\pm\} }\eta  \mathop{\sum}_{\substack{\alpha (M_1^{b_1})  }} \sum_{p_1p_2=M_2^{1+b_2}} \sum_{\substack{\gamma_1 \shortmod{p_1} \\ (\gamma_1(\gamma_1+p_2),p_1)=1}} \sum_{ad=c_0}\sum_{\substack{\gamma \shortmod{d} \\ (\gamma(\gamma+a),d)=1}}  \mathcal{C}_1^{\eta}(m, . ) \mathcal{C}_2^{\eta}(m^\prime, .) ,
    \end{align}
    where 
    \[ \mathcal{C}_1^{\eta}(m, . ) = e\left(\frac{   m \alpha}{M_1^{b_1}} \right) e\left(\frac{\eta m M_1^{b_1}}{Mc }   \right) \e{\eta \overline{M_1c_0} \overline{\gamma_1} m} {p_1}  \e{\eta \overline{M_1^{1+b_1}M_2^{1+b_2} } \overline{\gamma} m }{d}, \]
and 
    
     \[ \mathcal{C}_2^{\eta}(m^\prime, . ) = e\left(\frac{   - m^\prime \alpha}{M_1^{b_1}} \right) e\left(\frac{\eta m^\prime M_1^{b_1}}{Mc }   \right) \e{- \eta (\overline{\gamma_1+p_2})  \overline{M_1c_0} m^\prime}{p_1}  \e{ - \eta (\overline{\gamma+a})  \overline{M_1^{1+b_1} M_2^{1+b_2}} m^\prime}{d}    \]

\end{lemma}
\begin{proof} We have 
    \begin{align*}
    \mathcal{C}_{Mc}(m,m^\prime)=&= \frac{1}{2}\sideset{}{^\star}{\sum}_{\alpha (Mc)}  e_{Mc}(\alpha m+\overline{\alpha} m^\prime) \sum_{\psi(M_1)} (1 \pm{}\psi(-1))\psi(\alpha) \\
  &= \frac{M_1-1}{2} \sum_{\pm}\pm\sideset{}{^\star}{\sum}_{\alpha (Mc)}\delta(\alpha=\pm 1 \, \mathrm{mod}\, M_1) e_{Mc}(\alpha m+\overline{\alpha} m^\prime) \\
  &= \frac{M_1-1}{2} \sum_{\pm}\pm e\left(\frac{\pm\overline{M_2c}(m+m^\prime)}{M_1} \right) S(\overline{M_1} m, \overline{M_1}m^\prime;M_2c).
\end{align*}
To the first factor we apply reciprocity 
\[ e\left(\frac{\pm\overline{M_2c}(m+m^\prime)}{M_1} \right)= e\left(\frac{\mp\overline{M_1}(m+m^\prime)}{M_2c} \right) e\left(\frac{\pm(m+m^\prime)}{Mc} \right).\]
We write $c$  as follows: $ c=c_0M_1^{b_1} M_2^{b_2} $, $ 0 \leq b_1, b_2 \ll  \log C $.   
Thus we have 
\begin{align*}
    \mathcal{C}_{Mc}(m,m^\prime)=&= \frac{1}{2}\sideset{}{^\star}{\sum}_{\alpha (Mc)}  e_{Mc}(\alpha m+\overline{\alpha} m^\prime) \sum_{\psi(M_1)} (1-\psi(-1))\psi(\alpha) \\
  &= \frac{M_1-1}{2} \sum_{\pm}\pm\sideset{}{^\star}{\sum}_{\alpha (Mc)}\delta(\alpha=\pm 1 \, \mathrm{mod}\, M_1) e_{Mc}(\alpha m+\overline{\alpha} m^\prime) \\
  &= \frac{M_1-1}{2} \sum_{\pm}\pm \mathop{\sumstar}_{\substack{\alpha (M_1^{1+b_1})  \\  \alpha = \pm 1 (M_1)}} e\left(\frac{\overline{M_2^{1+b_2}c_0}( \alpha m+\overline{\alpha } m^\prime)}{M_1^{1+b_1}} \right) S(\overline{M_1^{1+b_1}} m, \overline{M_1^{1+b_1}}m^\prime;M_2^{1+b_2}c_0).
\end{align*}
The sum over $\alpha$ can be evaluated as follows: 
\[ \mathop{\sumstar}_{\substack{\alpha (M_1^{1+b_1})  \\  \alpha = \pm 1 (M_1)}} e\left(\frac{ \overline{M_2^{1+b_2}c_0}( \alpha m+\overline{\alpha } m^\prime)}{M_1^{1+b_1}} \right) = e\left(\frac{\pm\overline{M_2^{1+b_2}c_0}(m+m^\prime)}{M_1} \right) \mathop{\sum}_{\substack{\alpha (M_1^{b_1})  }} e\left(\frac{ \overline{M_2^{1+b_2}c_0} (  m- m^\prime) \alpha}{M_1^{b_1}} \right) \]

To the first factor we apply reciprocity 
\[ e\left(\frac{\pm\overline{M_2^{1+b_2}c_0}(m+m^\prime)}{M_1} \right)= e\left(\frac{\mp\overline{M_1}(m+m^\prime)}{M_2^{1+b_2}c_0} \right) e\left(\frac{\pm(m+m^\prime)}{McM_1^{-b_1} } \right).\]
Thus we get 
\begin{align}
        \mathcal{C}_{Mc}(m,m^\prime)= \frac{M_1-1}{2} \sum_{\pm}\pm e\left(\frac{\pm(m+m^\prime)}{Mc M_1^{-b_1}}   \right) \mathop{\sum}_{\substack{\alpha (M_1^{b_1})  }} e\left(\frac{  (  m- m^\prime) \alpha}{M_1^{b_1}} \right) \mathcal{C}_{M_2^{1+b_2}c_0}^{\pm}(m, m^\prime),
    \end{align}
    where 
    \begin{align*}
    \mathcal{C}_{M_2^{1+b_2}c_0}^{\pm }(m, m^\prime) &=  e\left(\frac{\mp\overline{M_1}(m+m^\prime)}{M_2^{1+b_2}c_0} \right)  S(\overline{M_1^{1+b_1}} m, \overline{M_1^{1+b_1}}m^\prime;M_2^{1+b_2}c_0)  
\end{align*}
Next we analyse the above chararacter sum using Lemma  \ref{matt lemma}. Indeed, for $\eta=\pm$, we have
\[  \mathcal{C}_{M_2^{1+b_2}c_0}^{\eta}(m, m^\prime) = \mathcal{C}_{ M_2^{1+b_2}}^{\eta}( m, m^\prime)\mathcal{C}_{c_0}^{\eta}(  m,  m^\prime)  \]
where $  \mathcal{C}_{ M_2^{1+b_2} }^{\eta}( \cdots)$ (resp. $\mathcal{C}_{c_0 }^{\eta}( \cdots)$ ) denote the contribution of $ \mathcal{C}_{ M_2^{1+b_2} c_0 }^{\eta}( \cdots)$ corresponding to the modulus   $M_2^{1+b_2}$ (resp. $c_0$). To these character sums we apply Lemma  \ref{matt lemma} to see that 
\begin{align*}
    \mathcal{C}_{M_2^{1+b_2}c_0}^{\eta}(m, m^\prime) 
    &=  \sum_{p_1p_2=M_2^{1+b_2}} \sum_{\substack{\gamma_1 \shortmod{p_1} \\ (\gamma_1(\gamma_1+p_2),p_1)=1}} \sum_{ad=c_0}\sum_{\substack{\gamma \shortmod{d} \\ (\gamma(\gamma+a),d)=1}}  \mathfrak{C}_1^{\eta}(m, \gamma,\gamma_1, ad) \mathfrak{C}_2^{\eta}(m^\prime, \gamma, \gamma_1, ad),
\end{align*}
where 
\[ \mathfrak{C}_1^{\eta}(\cdots)=\mathfrak{C}_1^{\eta}(m, \gamma,\gamma_1, ad)=  \e{\eta \overline{M_1c_0} \overline{\gamma_1} m} {p_1}  \e{\eta \overline{M_1^{1+b_1}M_2^{1+b_2} } \overline{\gamma} m }{d},\]
and 
\[ \mathfrak{C}_2^{\eta}(\cdots)=  \mathfrak{C}_2^{\eta}(m^\prime, \gamma, \gamma_1, ad)= \e{- \eta (\overline{\gamma_1+p_2})  \overline{M_1c_0} m^\prime}{p_1}  \e{ - \eta (\overline{\gamma+a})  \overline{M_1^{1+b_1} M_2^{1+b_2}} m^\prime}{d}.\]
 The lemma follows.
\end{proof}
\subsection{Mellin inversion}
In this section we separate $m$ and $m^\prime$ in the weight function $W^{\beta}\left(\frac{2 \sqrt{mm^\prime}}{Mc}\right)$ using the Mellin inversion trick. To this end,  we attach a smooth, compactly--supported weight function $w_1(\sqrt{mm^\prime}/N)$ that 
takes the value  1 for all $m$ and $m^\prime$ in the support of $V$. Using the Mellin--inversion, we have
\begin{align}\label{mellin 2}
     w_1 \left(\frac{\sqrt{mm^\prime}}{N}\right) W^{\beta}\left(\frac{2 \sqrt{mm^\prime}}{Mc}\right)=\frac{1}{2\pi}\int_{\mathbb{R}} F(iy) \left(\frac{N}{\sqrt{mm^\prime}} \right)^{iy} \mathrm{d}y,
\end{align}
where 
\[F(iv)=F_{c,M,N}(iv)=\int_0^{\infty} w_1(x) W^{\beta}\left( \frac{2 Nx}{Mc}\right) x^{iv} \frac{\mathrm{d}x}{x}.\]
If $|y| > K:=(1+ \frac{N}{CM})M^{\epsilon}$, then repeated integration by parts shows $F (iy)$  is $O(M^{-A})$ for any $A>0$.  We consider three cases: small $C$, i.e., $ C \ll M^{-\epsilon}N/M$,  large $C$, i.e., $ C \gg M^{\epsilon}N/M$ and  generic $C$, i.e., $  M^{-\epsilon}N/M \ll C \ll M^{\epsilon}N/M$. 

First assume that $C \ll M^{-\epsilon}N/M$. For $|y| \leq K$, we write $W^\beta(z)=\frac{1}{{z^{1/2}}} e(\beta z) W_1(z) $ with $z^jW_1^{(j)}(z) \ll_k 1 $, $j \geq  0$.  
Then  we have 
\[F(iv)= \frac{\sqrt{Mc}}{\sqrt{N}}\int_0^{\infty} w_1(x) x^{-1/2}e\left( \beta \frac{2 Nx}{Mc}\right) W_1\left(  \frac{2 Nx}{Mc}\right) x^{iv} \frac{\mathrm{d}x}{x}.\]
We apply the second derivative bound  to see that 
\begin{align}
    F(iy) \ll M^\epsilon {CM/N}.
\end{align}

Next we assume that  $ M^{\epsilon}N/M \ll C \ll M^{2025}N/M $. In this case we use the bound $W^{\beta}(z) \ll z^{k-1} \leq z^2, z<1, $  to get $ F(iy) \ll M^\epsilon {N^2}/M^2C^2. $

In the generic case $  M^{-\epsilon}N/M \ll C \ll M^{\epsilon}N/M$, we have $F(iy) \ll M^{\epsilon}. $
We summarize the above discussion in the following lemma. 
\begin{lemma}\label{separ}
   Let $K:=(1+ \frac{N}{CM})M^{\epsilon}.$ Then  we have 
    \begin{align}\label{mellin}
     w_1 \left(\frac{\sqrt{mm^\prime}}{N}\right) W^{\beta}\left(\frac{2 \sqrt{mm^\prime}}{Mc}\right)=\frac{1}{2\pi}\int_{-K}^{K} F(iy) \left(\frac{N}{\sqrt{mm^\prime}} \right)^{iy} \mathrm{d}y +O(M^{-2025}),
\end{align}
    with 
    \begin{align}\label{second derivative}
    F(iy) \ll M^\epsilon \min \{ {CM/N}, N^2/M^2C^2\}.
\end{align}
\end{lemma}
\section{AM--GM inequality}
On plugging  in Lemma \ref{separ of char} and Lemma \ref{separ} into  \eqref{OD}, we see that $ \mathrm{OD}_{\pi}(\ell,C, N)$ (see \eqref{OD coprime}) is given by 
\begin{align*}
       \mathrm{OD}_{\pi}^{}(\cdots) \ll & M^{\epsilon}\min\left\{\frac{M_1}{{N}}, \frac{M_1N^2}{C^3M^3}\right\} \sum_{\pm } \sum_{0 \leq b_1,  b_2 \ll \log C } \\
        & \times  \mathop{\sumstar_{ \substack{c_0 \sim \frac{C}{M_1^{b_1}M_2^{b_2}}
       } } }   \mathop{\sum}_{\substack{\alpha (M_1^{b_1})  }} \sum_{p_1p_2=M_2^{1+b_2}} \sum_{\substack{\gamma_1 \shortmod{p_1} \\ (\gamma_1(\gamma_1+p_2),p_1)=1}} \sum_{ad=c_0}\sum_{\substack{\gamma \shortmod{d} \\ (\gamma(\gamma+a),d)=1}}   \\
       & \times \int_{-K}^K \Big|\sum_m  \frac{ a_\ell(m)}{m^{iy/2}}   \mathcal{C}_1^{\pm}(\cdots)    \Big|  
        \Big|\sum_{m^\prime} \frac{\overline{a_\ell(m^\prime)}}{{m^\prime}^{iy/2}}   \mathcal{C}_2^{\pm}(\cdots ) \Big| \,\mathrm{d}y,
\end{align*}
where $a_\ell(m)=\lambda_\pi (\ell, m)V(m/N)$ and $\star$ on the $c_0$-sum indicates $(c_0,M)=1$. Using the AM-GM inequality $ |A||B|\leq 
\frac{1}{2}|A|^2 + \frac{1}{2} |B|^2$, we see that 
\begin{align}\label{od after am-gm}
    \mathrm{OD}_{\pi}(\cdots) \ll  M^{\epsilon}\min\left\{\frac{M_1}{{N}}, \frac{M_1N^2}{C^3M^3}\right\}   \sum_{\pm}  \sum_{0 \leq b_1, b_2 \ll \log C }  (\mathrm{OD}_{\pi}^{(1, \pm)}(\cdots)+\mathrm{OD}_{\pi}^{(2, \pm)}(\cdots )),
\end{align}
where 
\begin{align*}
       \mathrm{OD}_{\pi}^{(1, \pm)}(\cdots) = \mathop{\sumstar_{ \substack{c_0 \sim \frac{C}{M_1^{b_1}M_2^{b_2}}
       } } }  \mathop{\sum}_{\substack{\alpha (M_1^{b_1})  }} \sum_{p_1p_2=M_2^{1+b_2}} \sum_{\substack{\gamma_1 \shortmod{p_1} \\ (\gamma_1(\gamma_1+p_2),p_1)=1}} \sum_{ad=c_0}\sum_{\substack{\gamma \shortmod{d} \\ (\gamma(\gamma+a),d)=1}}   
        \int_{-K}^K \Big|\sum_m  \frac{ a_\ell(m)}{m^{iy/2}}   \mathcal{C}_1^{\pm}(m,.)    \Big|^2
       \,\mathrm{d}y,
\end{align*}
and 
\begin{align*}
       \mathrm{OD}_{\pi}^{(2, \pm)}(\cdots) = \mathop{\sumstar_{ \substack{c_0 \sim \frac{C}{M_1^{b_1}M_2^{b_2}}
       } } }  \mathop{\sum}_{\substack{\alpha (M_1^{b_1})  }} \sum_{p_1p_2=M_2^{1+b_2}} \sum_{\substack{\gamma_1 \shortmod{p_1} \\ (\gamma_1(\gamma_1+p_2),p_1)=1}} \sum_{ad=c_0}\sum_{\substack{\gamma \shortmod{d} \\ (\gamma(\gamma+a),d)=1}}   
        \int_{-K}^K \Big|\sum_{m^\prime}  \frac{ \overline{a_\ell(m^\prime)}}{m^{\prime iy/2}}   \mathcal{C}_2^{\pm}(m^\prime,.)    \Big|^2
       \,\mathrm{d}y,
\end{align*}
We further proceed with $ \mathrm{OD}_{\pi}^{(2,+)}(\ell,C, N)$ as the analysis for the other terms is mutatis mutandis or even easier. Recall that 
   \[ \mathcal{C}_2^{+}(m^\prime, . ) = e\left(\frac{   - m^\prime \alpha}{M_1^{b_1}} \right)  e\left(\frac{m^\prime}{MM_2^{b_2}c_0} \right) \e{-  (\overline{\gamma_1+p_2})  \overline{M_1c_0} m^\prime}{p_1}  \e{ -  (\overline{\gamma+a})  \overline{M_1^{1+b_1} M_2^{1+b_2}} m^\prime}{d}    \]
On relaxing the conditions $(\gamma_1,p_1)=1$ and $(\gamma,d)=1$ by positivity, and then using the change of  variable $\gamma_1+p_2\rightarrow \gamma_1$, $\gamma+a \rightarrow \gamma$,  and then merging $\gamma$ and $\gamma_1$ (after some standard change of variables), we see that $  \mathrm{OD}_{\pi}^{(2, +)}(\cdots)  $ is bounded by
\begin{align*}
        \int_{-K}^K  \mathop{\sumstar_{ \substack{c_0 \sim \frac{C}{M_1^{b_1}M_2^{b_2}}
       } } }  \mathop{\sum}_{\substack{\alpha (M_1^{b_1})  }} \sum_{p_1p_2=M_2^{1+b_2}}  \sum_{ad=c_0}   
        \sumstar_{\substack{ \gamma \shortmod{dp_1} }  }\Big| \sum_{m^\prime}  \frac{ \overline{a_\ell(m^\prime)}}{m^{\prime iy/2}}  e\left(\frac{   - m^\prime \alpha}{M_1^{b_1}} \right)  e\left(\frac{ \gamma m^\prime}{p_1d} \right)  e\left(\frac{m^\prime}{MM_2^{b_2}c_0} \right)  \Big|^2.
\end{align*}  On applying a dyadic partition of unity to the $a$-sum, we arrive at 
\begin{align}\label{OD after a}
     \mathrm{OD}_{\pi}^{(2, +)}(\cdots)  \ll  \sup_{A \ll C}  \sum_{a \sim A} \sum_{b_1 , b_2 \ll  \log C } G_{}^{b_1,b_2}(\ell, A,C,N),
\end{align}
where 
\[ G_{}^{b_1,b_2}(\ell, A,C,N)= \sum_{p_1 \mid M_2^{1+b_2}}  G_{p_1}^{b_1}(\ell, A,C,N),\]
\begin{align}\label{Gp}
       G_{p_1}^{b_1}(\ell, A,C,N) =  \mathop{\sumstar_{ \substack{d \sim \frac{C}{AM_1^{b_1}M_2^{b_2}}
       } } }   \mathop{\sum}_{\substack{\alpha (M_1^{b_1})  }}  \sumstar_{\substack{ \gamma \shortmod{dp_1} }  } \int_{-K}^K \Big| \sum_{m^\prime}  \frac{ \overline{a_\ell(m^\prime)}}{m^{\prime iy/2}}  e\left(\frac{   - m^\prime \alpha}{M_1^{b_1}} \right)  e\left(\frac{ \gamma m^\prime}{p_1d} \right)  e\left(\frac{m^\prime}{MM_2^{b_2}ad} \right)  \Big|^2 \mathrm{d}y,
\end{align}
\section{Analysis of $ G_{p_1}^{b_1}(\ell, A,C,N)$   }
\subsection{$b_1 \geq 1$ }
We first analyze $ G_{p_1}^{b_1}(\ell, A,C,N)$ for $b_1 \geq 1$ which is easier to deal with.
\begin{lemma}\label{easy bound}
    For  $b_1 \geq 1,$ we have 
    \[  G_{p_1}^{b_1}(\ell, A,C,N) \ll M^\epsilon \frac{C^2}{A} (N+ \frac{N^2}{M_1^2}),\]
    and thus 
    \[  \mathrm{OD}_{\pi}^{(2, +)}(\cdots)  \ll M^\epsilon  {C^2}{} (N+N^2/M_1^2).   \]
\end{lemma}
\begin{proof}
For $b_1 \geq 1,$ we have  { $C \gg {M^{-\epsilon} N}/{M}$}.  Indeed  
    $$ C \ll {M^{-\epsilon} N}/{M} \ \mathrm{and} \ M_1|c \implies  M_1 \leq M^{-\epsilon}N/M  \ll M^{-\epsilon } \sqrt{M} \implies M_1^{} \ll M^{ -2\epsilon}M_2.$$
On executing the sum over $d$,  $\gamma(d)$ and $y$-integral  trivially in \eqref{Gp}, we see that 
  
 \begin{align}
       G_{p_1}^{b_1}(\ell, A,C,N) \ll  M^\epsilon \frac{C^2  }{A^2 M_1^{2b_1} M_2^{2b_2} }  \mathop{\sum}_{\substack{\alpha (M_1^{b_1}p_1)  }}  \Big| \sum_{m\sim N}  {a_m}  e\left(\frac{    m \alpha}{M_1^{b_1}p_1} \right)     \Big|^2,
\end{align}
for  some  coefficients $a_m$  satisfying $\sum_{m \sim N} |a_m|^2 \ll_\pi N$. On applying  Lemma  \ref{C=1} with $D=M_1^{b_1}p_1$, we get 
\begin{align*}
       G_{p_1}^{b_1}(\ell, A,C,N) \ll  M^\epsilon \frac{C^2}{A^2 M_1^{2b_1} M_2^{2b_2} } ( M_1^{b_1}p_1+N)N \leq M^\epsilon \frac{C^2}{A} ( \frac{M_2}{M_1}N+ \frac{N^2}{M_1^2})
\end{align*}

The second part follows upon applying the above bound into \eqref{OD after a}. 
 \end{proof}

\subsection{Analysis of $ G_{p_1}^0(\ell, A,C,N)$}
We now analyse $ G_{p_1}^0(\ell, A,C,N)$, which is the heart of the matter. 
Recall that 
\begin{align}\label{pre voronoi}
     G_{p_1}^0(\ell, A,C,N) =  \mathop{\sumstar_{ \substack{d \sim \frac{C}{A M_2^{b_2}}
       } } }   \sumstar_{\substack{ \gamma \shortmod{dp_1} }  } \int_{-K}^K | H(d, \gamma,a,y)  |^2 \mathrm{d}y,
\end{align}
where
\[ H(d, \gamma,a,y)= \sum_{m=1}^\infty  \frac{ \overline{a_\ell(m )}}{m^{ iy/2}}    e\left(\frac{ \gamma m }{p_1d} \right)  e\left(\frac{m }{MM_2^{b_2}ad} \right),\]
with $a_\ell(m)=\lambda_\pi (\ell, m)V(m/N)$. 
We will make use of the $GL(3)$ coefficients by analysing the above sum via the $GL(3)$ Voronoi summation. 
On applying Lemma \ref{gl3voronoi} to $H(d, \gamma,a,y) $ with $g(m)=V\left( \frac{m}{N}\right) e\left(\frac{m }{MM_2^{b_2}ad} \right)m^{-iy/2}$ we get 
\begin{align}\label{after gl3 voronoi}
    H(d, \gamma,a,y)=p_1d \sum_{\pm} \sum_{n_1|p_1d\ell} \sum_{n_2=1}^\infty \frac{\overline{\lambda_\pi(n_1,n_2)}}{n_1n_2}S(\ell \overline{\gamma}, \mp n_2;p_1d\ell /n_1 )G_{\pm} \left( \frac{n_1^2n_2}{p_1^3d^3 \ell}\right),
\end{align}
where 
\begin{align*}
G_{\pm}(x) = \frac{1}{2 \pi i} \int_{(\sigma)} x^{-s} \, \gamma_{\pm}(s) \, \widetilde{g}(-s) \, \mathrm{d}s, \  \widetilde{g}(-s)= \int_0^\infty g(z) z^{-s-1} \mathrm{d}z.
\end{align*}
Next we analyse $G_\pm (x)$. 
\begin{lemma}\label{voronoi lemma}
	 Let $x=\frac{n_1^2n_2}{ p_1^3d^3\ell}$ and $N_2=\frac{M^\epsilon M_2^3 C^3\ell K^3 }{ A^3N}$.  Then 
	$G_{\pm}(x)$ is negligibly small unless  $n_2 \leq  N_{2}/n_1^2$, in  which  case, we have $G_{\pm} \left(x\right) = \left( Nx\right)^{1/2} I{(x)}+O(M^{-2025}) ,$ where $I(x)$ is an integral transform (see \eqref{integral of gl3})  satisfying $I(x) \ll \sqrt{K}$.
\end{lemma}
\begin{proof}
	We have 
	\begin{align}\label{Gy}
		G_{\pm}(x) &= \frac{1}{2 \pi i} \int_{(\sigma)} x^{-s} \, \gamma_{\pm}(s) \, \int_{0}^{\infty} V \left( \frac{z}{N}\right) e\left(\frac{z }{MM_2^{b_2}ad} \right)z^{-iy/2}  z^{-s-1} \, \mathrm{d}z\, \mathrm{d}s  \notag\\
		&= \frac{1}{2 \pi } \int_{-\infty}^{\infty} (Nx)^{-\sigma-i\tau} \, \gamma_{\pm}(\sigma+i\tau) \, \int_{0}^{\infty}  V \left( {z}{}\right) e\left(\frac{zN }{MM_2^{b_2}ad} \right)(Nz)^{-iy/2} z^{-\sigma-1-i\tau} \, \mathrm{d}z\, \mathrm{d}\tau.
	\end{align}
	On applying integration by parts, we conclude that the $z$-integral is negligibly small unless $|\tau| \leq  K.$
	By the  Stirling formula, we have  
	$$\gamma_{\pm}(\sigma+i\tau) \ll_{\pi,\sigma} (1+|\tau|)^{3(\sigma+1/2)}.$$
	Thus, on plugging this bound into \eqref{Gy}, we get 
	\begin{align}
		G_{\pm}(x)  \ll_{\pi} K^{5/2}  \left(\frac{Nx}{K^3 }\right)^{-\sigma}. 
	\end{align}
	  On taking   $\sigma $ sufficiently large (towards $\infty$) we see that $G_{\pm}(x) $ is negligibly small unless
	\begin{align}\label{N0}
		Nx \ll M^{\epsilon}K^3  \iff n_1^2n_2 \leq M^{\epsilon} \frac{p_1^3 d^3\ell K^3 }{N} \leq M^\epsilon \frac{p_1^3 C^3\ell K^3 }{M_2^{3b_2} A^3N}.
	\end{align}
	In the above range, we move the contour to $\sigma=-1/2$ to get 
	\begin{align}
		G_{\pm} \left(x\right) = \left( Nx\right)^{1/2} I{(x)}+O(M^{-2025}) ,
	\end{align}
	where 
	\begin{align}\label{integral of gl3}
		I(x):=	\frac{1}{2 \pi } \int_{|\tau| \leq K}  \left( Nx\right)^{-i\tau} \, \gamma_{\pm}(-1/2+i\tau)\widetilde{W}(\tau,y)  \mathrm{d}\tau ,
	\end{align}
	and 
	\begin{align}
		\widetilde{W}(\tau,y )= \int_{0}^{\infty}  V \left( {z}{}\right) e\left(\frac{zN }{MM_2^{b_2}ad} \right)(Nz)^{-iy/2} z^{-1/2-i\tau} \, \mathrm{d}z
	\end{align}

By the second derivative bound, we get 
\begin{align}\label{second der}
    \widetilde{W}(\tau,y) \ll 1/\sqrt{K}.
\end{align}
The lemma follows. 
\end{proof} 

On applying \eqref{after gl3 voronoi}  into \eqref{pre voronoi}  we see that 
\begin{align}\label{n_2 dyadic}
&G_{p_1}^0(\ell, A,C,N) \ll  \notag \\
       &\mathop{\sumstar_{ \substack{d \sim \frac{C}{A M_2^{b_2}}
       } } }   \sumstar_{\substack{ \gamma \shortmod{dp_1} }  }   \int_{-K}^K \Big| \sum_{\pm}p_1d \mathop{ \sum_{ \substack{n_1|p_1d\ell  } }}  \sum_{n_2 \leq  \frac{N_2}{n_1^2}  } \frac{\overline{\lambda_\pi(n_1,n_2)}}{ n_1{n_2} }S(\ell \overline{\gamma}, \mp n_2;p_1d\ell /n_1 )G_\pm \left( \frac{n_1^2n_2}{p_1^3d^3 \ell}\right) \Big|^2 \mathrm{d}y.
\end{align}
We consider two cases; small $n_2$ and large $n_2$.   
\subsubsection{Small $n_2$: $ n_1^2n_2N/(p_1d)^3 \ell \leq M^\epsilon$}\label{small n}  Let $N_0= M^\epsilon \frac{M_2^3 C^3\ell  }{ A^3N}$.
In this case, we apply  Lemma \ref{voronoi lemma} into \eqref{n_2 dyadic}, to see that $ G_{p_1}^0(\ell, A,C,N)$ is dominated by 
\begin{align*}
       &\mathop{\sumstar_{ \substack{d \sim \frac{C}{A M_2^{b_2}}
       } } }   \sumstar_{\substack{ \gamma \shortmod{dp_1} }  }  \frac{ {N}}{p_1d  {\ell}} \int_{-K}^K \Big| \sum_{\pm} \sum_{n_1|p_1d\ell} \sum_{n_2 } \frac{\overline{\lambda_\pi(n_1,n_2)}}{ \sqrt{n_2} }S(\ell \overline{\gamma}, \mp n_2;p_1d\ell /n_1 )I \left( \frac{n_1^2n_2}{p_1^3d^3 \ell}\right) \Big|^2 \mathrm{d}y \\
        \ll & \mathop{\sumstar_{ \substack{d \sim \frac{C}{A M_2^{b_2}}
       } } }   \sumstar_{\substack{ \gamma \shortmod{dp_1} }  }  \frac{{N} K}{p_1d {\ell}} \sum_{n_1|p_1d\ell} \int_{-K}^K  \Big|   \sum_{ n_2 \leq N_0/n_1^2} \frac{\overline{\lambda_\pi(n_1,n_2)}}{ \sqrt{n_2}  } n_2^{{-i \tau}}S(\ell \overline{\gamma}, \mp n_2;p_1d\ell /n_1 )    \Big|^2  \mathrm{d}\tau.
\end{align*}
In the last inequality, we applied Cauchy inequality to the $\tau$-integral and the $n_1$-sum and used the bound $$ \int_{-K}^K |\widetilde{W}(\tau,y)|^2 \mathrm{d} \tau \ll 1.$$
from \eqref{second der}. 
Let $e=(p_1d,n_1)$. Write $e=e_1e_2,$ with $e_1|p_1,$  $e_2|d,$ $(e_1,e_2)=1$.  By change of variable, $d \rightarrow e_2d$, $p_1 \rightarrow e_1p_1$ and $n_1 \rightarrow en_1$, we see that (after dividing the sum over $n_1$ dyadically: $n_1 \sim N_1$ )
\begin{align}\label{n1 dyadic}
     \sum_{p_1| M_2^{1+b_2}} G_{p_1}^0(\ell, A,C,N) &\ll   \sum_{p_1| M_2^{1+b_2}} \mathop{\sum_{\substack {e_1e_2 \ll N_1 \\ e_1\mid M_2^{1+b_2} }   }} \mathop{\sumstar_{ \substack{d \sim \frac{C}{A M_2^{b_2} e_2}
       } } }  \sum_{\substack{ \gamma \shortmod{edp_1} \\ (\gamma, edp_1)=1 }   }  \frac{ {N} K}{ep_1d  {\ell}} \mathop{\sum_{ \substack{n_1|\ell \\ (n_1,p_1d)=1 \\ 
       en_1 \asymp N_1} }} \notag \\
     &  \times \int_{-K}^K \Big|   \sum_{  n_2 \leq N_0/n_1^2 } \frac{\overline{\lambda_\pi(en_1,n_2)}}{ \sqrt{n_2}  } n_2^{{-i \tau}}S(\ell \overline{\gamma}, \mp n_2;p_1d\ell /n_1 )    \Big|^2  \mathrm{d}\tau.
\end{align}
Write  $\ell/n_1=rs, $ $r \mid (p_1d)^\infty$ and $(s,p_1d)=1$. Thus 
\begin{align}\label{Gp3}
       \sum_{p_1| M_2^{1+b_2}} G_{p_1}^0(\ell, A,C,N)  &\ll   \sum_{p_1| M_2^{1+b_2}}\mathop{\sum_{\substack {e_1e_2 \ll N_1 \\ e_1\mid M_2^{1+b_2} }   }}  \mathop{\sumstar_{ \substack{d \sim \frac{C}{A M_2^{b_2} e_2}
       } } }    \frac{ {N} K}{ep_1d  {\ell}} \mathop{\sum_{ \substack{n_1rs= \ell \\ (n_1s,p_1d)=1 \\ r | (p_1d)^\infty \\ en_1 \asymp N_1 }}} \sum_{\substack{ \gamma_1 \shortmod{e} }   }  \int_{-K}^K  \mathcal{L}_1 \,  \mathrm{d} \tau \notag \\
        &=\sum_{p_1| M_2^{1+b_2}} \mathop{\sum_{\substack {e_1e_2 \ll N_1 \\ e_1\mid M_2^{1+b_2} }   }}  \mathop{\sumstar_{ \substack{d \sim \frac{C}{A M_2^{b_2} e_2}
       } } }   \frac{ {N} K}{p_1d  {\ell}} \mathop{\sum_{ \substack{n_1rs= \ell \\ (n_1s,p_1d)=1 \\ r | (p_1d)^\infty \\ en_1 \asymp N_1  }}}    \int_{-K}^K  \mathcal{L}_1 \,  \mathrm{d} \tau,
\end{align}
where 
\begin{align*}
    \mathcal{L}_1= \sum_{\substack{ \gamma \shortmod{dp_1} \\ (\gamma, dp_1)=1 }   } \Big|   \sum_{n_2 \leq N_0/(en_1)^2}\frac{\overline{\lambda_\pi(en_1,n_2)}}{ \sqrt{n_2}  } n_2^{{-i \tau}}S(rsn_1 \overline{\gamma}, \mp n_2;p_1drs)    \Big|^2.
\end{align*}

\begin{lemma}\label{kloosterman simple}
   For any sequence $\alpha(n)$ of complex numbers, we have 
   \[   \mathcal{L}= \sum_{\substack{ \gamma \shortmod{dp_1} \\ (\gamma, dp_1)=1 }   } \Big|   \sum_{n \leq N} \alpha(n) S(rsn_1 \overline{\gamma}, \mp n;p_1drs)    \Big|^2 \ll  p_1dr^2s \sum_{\substack{ h \shortmod{dp_1 s} \\ (h, dp_1s)=1 }   } \Big|   \sum_{n \leq N/r } \alpha(rn) e\left( \frac{hn}{p_1ds}\right)     \Big|^2. \]
\end{lemma}
\begin{proof}
We follow arguments from \cite{MY11}[P. 22, 11.5]. 
    Using  the multiplicativity relation for Kloosterman sums, we obtain
    \[S(rsn_1 \overline{\gamma}, \mp n;p_1drs)= S(rn_1 \overline{\gamma s}, \mp n;p_1dr) S(0,  n;s). \]
    By the change of variable $n_1 \overline{\gamma s} \rightarrow \gamma$, we  arrive at 
    \begin{align*}
     \sum_{\substack{ \gamma \shortmod{dp_1} \\ (\gamma, dp_1)=1 }   } \Big|   \sum_{n} \alpha_1(n) S(r {\gamma}, \mp n;p_1dr)    \Big|^2, \ \alpha_1(n)=\alpha(n) S(0,n;s)
\end{align*}
We further compute 
\begin{align}\label{cong}
     \sum_{\substack{ \gamma \shortmod{dp_1}  }   } \Big|   \sum_{n} \alpha_1(n) S(r {\gamma}, \mp n;p_1dr)    \Big|^2=p_1d \sum_{n_1,n_2} \alpha_1(n_1) \overline{\alpha_1}(n_2) \mathop{\sumstar_{ \substack{h_1, h_2(\mathrm{mod}\, p_1dr)  \\ h_1 \equiv h_2 (\mathrm{mod}\, p_1d)
       } } } e\left( \frac{\mp n_1 h_1 \pm n_2h_2 }{p_1dr}\right)
\end{align}
We  change of variable $h_i=y+p_1d z_i, $ $i=1,2,$ where $y$ runs modulo $p_1d$ and $z_i$ runs modulo $r$. Since $r |(p_1d)^\infty,$ the condition $(h_i,p_1dr)=1$ is equivalent to $(y,b)=1$. The sum over $z_i$ vanishes unless $r|n_i,$ in which case the sum is $r$. Thus \eqref{cong} equals
\[p_1dr^2\sum_{r |n_1,n_2} \alpha_1(n_1) \overline{\alpha_1}(n_2) \sumstar_{y\, \mathrm{mod}\, p_1d} e \left( \frac{y (\mp\frac{n_1}{r} \pm\frac{n_2}{r})}{p_1d}\right)= p_1dr^2 \sumstar_{y\, \mathrm{mod}\, p_1d} \Big| \sum_{r|n} \alpha_1(n) e \left( \frac{\mp y\frac{n}{r}}{p_1d} \right)\Big|^2 .  \]
Hence 
\[  \sum_{\substack{ \gamma \shortmod{dp_1} \\ (\gamma, dp_1)=1 }   } \Big|   \sum_{n \leq N} \alpha(n) S(rsn_1 \overline{\gamma}, \mp n;p_1drs)    \Big|^2 \leq p_1dr^2 \sumstar_{y\, \mathrm{mod}\, p_1d} \Big| \sum_{r|n} \alpha(n) S(0,n;s) e \left( \frac{\mp y\frac{n}{r}}{p_1d} \right)\Big|^2. \]
By Cauchy's inequality, for any complex coefficients $b_\ell$, we have 
\[|\sum_\ell b_\ell S(0, \ell;s)|^2\leq s \sumstar_{h \, \mathrm{mod}\, s} |\sum_\ell b_\ell e(\frac{h\ell}{s})|^2. \]
Hence using $S(0,\ell;s)=S(0,\ell/r;s)$ 
\[\mathcal{L} \leq p_1dr^2s  \sumstar_{h \, \mathrm{mod}\, s} \ \sumstar_{y\, \mathrm{mod}\, p_1d}\Big| \sum_{r|n} \alpha(n) e \left( \frac{ h\frac{n}{r}}{s} \right) e \left( \frac{\mp y\frac{n}{r}}{p_1d} \right)\Big|^2   \]
We get the lemma using the Chinese remainder theorem and changing the variables $n \rightarrow nr$.
\end{proof}

On plugging Lemma \ref{kloosterman simple} into \eqref{Gp3} we get 
\begin{align*}
     \sum_{p_1| M_2^{1+b_2}} G_{p_1}^0(\ell, A,C,N )  &\ll  NK  \sum_{p_1| M_2^{1+b_2}} \mathop{\sum_{\substack {e_1e_2 \ll N_1 \\ e_1 \mid M_2^{1+b_2} }   }}\mathop{\sumstar_{ \substack{d \sim \frac{C}{A M_2^{b_2} e_2}
       } } }    \mathop{\sum_{ \substack{n_1rs= \ell \\ (n_1s,p_1d)=1 \\ r | (p_1d)^\infty \\ en_1 \asymp N_1 }}} \frac{1}{n_1}    \int_{-K}^K  \\
       & \times \sum_{\substack{ h \shortmod{dp_1 s} \\ (h, dp_1s)=1 }   } \Big|   \sum_{rn_2 \leq  N_0/(en_1)^2}\frac{\overline{\lambda_\pi(en_1,rn_2)}}{ \sqrt{ n_2}  } n_2^{{-i \tau}} e\left( \frac{hn_2}{p_1ds}\right)     \Big|^2 \mathrm{d} \tau
\end{align*}
Thus 
\begin{align}\label{before large sieve AV}
      \sum_{p_1 | M_2^{1+b_2}}G_{p_1}^0(\ell, A,C,N)  &\ll  NK^2 \sum_{p_1 | M_2^{1+b_2}} \mathop{\sum_{\substack {e_1e_2 \ll N_1  \\ e_1|M_2^{1+b_2}  }   }}  \mathop{\sum_{ \substack{n_1rs= \ell \\ (n_1s,p_1d)=1 \\ r | (p_1d)^\infty \\ en_1 \asymp N_1 }}} \frac{1}{n_1} \mathcal{L}_2 ,
\end{align}
where 
\[ \mathcal{L}_2= \mathop{\sumstar_{ \substack{d \sim \frac{C}{A M_2^{b_2} e_2}
       } } } \sum_{\substack{ h \shortmod{dp_1 s} \\ (h, dp_1s)=1 }   } \Big|   \sum_{rn_2 \leq N_0/(en_1)^2} \frac{\overline{ \lambda_\pi(en_1,rn_2)}}{ \sqrt{ n_2}  } n_2^{{-i \tau_0}} e\left( \frac{hn_2}{p_1ds}\right)     \Big|^2,\]
for some $\tau_0$.

\begin{lemma}\label{small n bound}
For $ n_1^2n_2N/(p_1d)^3 \ell \leq M^\epsilon$,    we have 
    \begin{align*}
        \mathop{\sum}_{\substack{p_1 \mid M_2^{1+b_2} }}  G_{p_1}^{0}(\ell, A,C,N) \ll \frac{NK^2 {\ell^2} C^2 M_2}{A^2    }  \left(1 +\frac{C M_2^2}{ N}\right).
    \end{align*}
\end{lemma}
\begin{proof}
 By  Lemma \eqref{large sieve with level}, we have
\begin{align}\label{l2 bound}
    \mathcal{L}_2 &\ll M^{\epsilon}\left( \frac{p_1sC^2 }{A^2 M_2^{2b_2}e_2^2 } +\frac{N_0}{r(en_1)^2}\right) \sum_{n_2} \frac{{ |\lambda_\pi(en_1,rn_2)|^2}}{{ n_2}  } \notag\\
    &\ll M^{\epsilon}\left( \frac{M_2sC^2 }{A^2 M_2^{b_2}e_2^2 } +\frac{M_2^3C^3 \ell}{ r A^3N(en_1)^2}\right) \sum_{n_2} \frac{{ |\lambda_\pi(en_1,rn_2)|^2}}{{ n_2}  }.
\end{align}
The contribution of  the first term above to \eqref{before large sieve AV} is bounded by 
 \begin{align*}
     &   \sum_{p_1|M_2^{1+b_2}}NK^2 \mathop{\sum_{\substack {e_1e_2 \ll N_1  \\e_1 \mid M_2^{1+b_2} }   }} \mathop{\sum_{\substack {en_1  \asymp N_1 }   }} \sum_{s \leq \ell} \sum_{r \ll \ell/s n_1}   \frac{1}{n_1}  \frac{M_2 C^2s  r}{ A^2M_2^{b_2}e_2^2  } \sum_{rn_2 \leq N_0/(en_1)^2}\frac{ { |\lambda_\pi(en_1,rn_2)|^2 }}{ { rn_2}  }.
\end{align*}
We change the variable $rn_2 \rightarrow n_2$  followed by executing the sum over $r$ to arrive at 
\begin{align*}
    &  \sum_{p_1|M_2^{1+b_2}}NK^2 \mathop{\sum_{\substack {e_1e_2 \ll N_1  \\e_1 \mid M_2^{1+b_2} }   }}  \mathop{\sum_{\substack {en_1  \asymp N_1 }   }} \sum_{s \leq \ell}    \frac{\ell^2}{n_1^3 s} \frac{C^2 M_2 }{A^2  M_2^{ b_2} e_2^2 }  \sum_{n_2 \leq  N_0/(en_1)^2}\frac{ { |\lambda_\pi(en_1,n_2)|^2 }}{ { n_2}  }     \\
     &\ll  \sum_{p_1|M_2^{1+b_2}}NK^2 \mathop{\sum_{\substack {e_1e_2 \ll N_1  \\e_1 \mid M_2^{1+b_2} }   }}  \mathop{\sum_{\substack {en_1  \asymp N_1 }   }}     \frac{\ell^2}{n_1^3 } \frac{C^2 M_2}{A^2  M_2^{ b_2} e_2^2 }  \sum_{n_2 \leq  N_0/(en_1)^2}\frac{ { |\lambda_\pi(en_1,n_2)|^2 }}{ { n_2}  }  .
\end{align*}
We change the variable $en_1 \rightarrow n_1$ to arrive at 
\begin{align*}
      &  \sum_{p_1|M_2^{1+b_2}}NK^2 \mathop{\sum_{\substack {e_1e_2 \ll N_1  \\e_1 \mid M_2^{1+b_2} }   }}       \frac{\ell^2 e^2}{N_1^2} \frac{C^2 M_2}{A^2  M_2^{ b_2} e_2^2 }  \times  \sum_{n_1 }\sum_{n_2 \ll N_0/n_1^2}  \frac{|\lambda_\pi(n_1,n_2)|^2 }{n_1 n_2}\\
      &\ll  \sum_{p_1|M_2^{1+b_2}}NK^2 \mathop{\sum_{\substack {e_1e_2 \ll N_1  \\e_1 \mid M_2^{1+b_2} }   }}      \frac{\ell^2  e_1^2}{N_1^2} \frac{C^2 M_2 }{A^2  M_2^{ b_2}  }  \\
      & \ll \sum_{p_1|M_2^{1+b_2}}NK^2 \mathop{\sum_{\substack {e_2 \ll N_1 }   }}  \sum_{e_1 \ll N_1/e_2}    \frac{\ell^2  e_1^2}{N_1^2} \frac{C^2 M_2}{A^2  M_2^{ b_2}  }  
      \ll  \frac{NK^2 {\ell^2} C^2 M_2}{A^2 M_2^{ b_2}   }
   \end{align*}

We have 
\begin{align*}
   & \frac{M_2^3C^3 \ell}{ r A^3N(en_1)^2}= \frac{M_2sC^2 }{A^2 M_2^{b_2}e_2^2 } \times  \frac{M_2^2 C M_2^{b_2} \ell }{ANe_1^2rsn_1^2 } .
\end{align*}
Note that (using $\ell=n_1rs$)
\[ \frac{M_2^2 C M_2^{b_2} \ell }{ANe_1^2rsn_1^2 }  \ll M^\epsilon \frac{ CM_2^2 M_2^{b_2}}{N} \]
Hence, following the above arguments,  the contribution of  the second  term in \eqref{l2 bound} to \eqref{before large sieve AV} is  bounded by  
$$ \frac{NK^2 {\ell^2} C^2 M_2}{A^2    }   \times \frac{CM_2^2}{N}.$$ The lemma follows. 

\end{proof}

\subsubsection{Large  $n_2$: $ n_1^2n_2N/(p_1d)^3 \ell > M^\epsilon$ } Using 
$n_1^2n_2 \leq N_2 =  M^{\epsilon/2}\frac{(p_1d)^3 \ell K^3 }{ N} $ from \eqref{N0} and  $n_1^2n_2N/(p_1d)^3 \ell > M^\epsilon $, we note that $K \gg M^{\epsilon}$, i.e., $C \ll M^{-\epsilon}N/M$ in this case. We use the representation (from Lemma \ref{GL3oscilation})
\begin{align*}
		G_{\pm}(x) &=  x \int_{0}^{\infty} \frac{V(z/N)}{z^{iy/2}}  e\left(\frac{ z}{MM_2^{b_2}ad} \right) \sum_{j=1}^{K_1} \frac{c_{j}({\pm}) e\left(\pm 3 (xz)^{1/3} \right)}{\left( xz\right)^{j/3}} \, \mathrm{d} z + O \left((xN)^{\frac{-K_1+5}{3}}\right) \\
        &=  \frac{xN}{N^{iy/2}} \int_{0}^{\infty} \frac{V(z)}{z^{iy/2}}  e\left(\frac{ zN }{MM_2^{b_2}ad} \right) \sum_{j=1}^{K_1} \frac{c_{j}({\pm}) e\left(\pm 3 (xzN)^{1/3} \right)}{\left( xzN\right)^{j/3}} \, \mathrm{d} z + O \left( \frac{1}{M^{\epsilon K_1}}\right)
	\end{align*}
    Thus choosing $K_1$ large enough and substituting  it   into \ref{after gl3 voronoi}, we see that 
\begin{align*}
    H(d, \gamma,a,y)=&\frac{p_1d}{ N^{iy/2}}  \sum_{j=1}^{K_1} \sum_{\pm} \sum_{\eta \in \{ \pm\}} c_j(\eta) \sum_{n_1|p_1d\ell} \sum_{n_2 \leq N_2/n_1^2} \frac{\overline{\lambda_\pi(n_1,n_2)}}{n_1n_2} \left(\frac{n_1^2 n_2N}{(p_1d)^3 \ell} \right)^{1-j/3} \\
    & \times S(\ell \overline{\gamma}, \mp n_2;p_1d\ell /n_1 )I_{\pm} \left( \frac{n_1^2n_2}{p_1^3d^3 \ell},y\right) +O(M^{- 2025}),
\end{align*}
where 
\begin{align}
    I_{\pm} \left( x,y\right)=\int_{0}^{\infty} {V(z )}{^{}}  e\left(\frac{ z N}{MM_2^{b_2}ad} \right)  \frac{  e\left(\pm 3 (xzN)^{1/3} \right)}{ z^{iy/2+j/3}} \, \mathrm{d} z.
\end{align}
On plugging this into \eqref{pre voronoi} and applying Cauchy's inequality to the sum over $j$, we have 
\begin{align}
     G_{p_1}^0(\ell, A,C,N) \ll   \sum_j    G_{p_1}^0(j),
\end{align}
where 
\begin{align}
       G_{p_1}^0(j)=&
      \mathop{\sumstar_{ \substack{d \sim \frac{C}{A M_2^{b_2}}
       } } }   \sumstar_{\substack{ \gamma \shortmod{dp_1} }  } \int_{-K}^K p_1^2d^2 \\ 
       & \times \Big| \sum_{n_1|p_1d\ell} \sum_{n_2 \leq N_2/n_1^2} \frac{\overline{\lambda_\pi(n_1,n_2)}}{n_1n_2} \left(\frac{n_1^2 n_2N}{(p_1d)^3 \ell} \right)^{1-j/3}S(\ell \overline{\gamma}, \mp n_2;p_1d\ell /n_1 )I_{\pm} \left( \frac{n_1^2n_2}{p_1^3d^3 \ell},y\right)  \Big|^2 \mathrm{d}y.
\end{align}
We will  focus on  $ G_{p_1}^0(1)$, as the  analysis for other $  G_{p_1}^0(j)$ is similar and in fact we get better estimates for $j \geq 2$ due to the presence of factor $(xN)^{j/3},$ $x=\frac{n_1^2 n_2}{(p_1d)^3 \ell}$ appearing in the denominator. 
Thus we have 
\begin{align}
       G_{p_1}^0(1) &\ll 
      \mathop{\sumstar_{ \substack{d \sim \frac{C}{A M_2^{b_2}}
       } } }   \sumstar_{\substack{ \gamma \shortmod{dp_1} }  }  \int_{-K}^K \frac{  N^{4/3}}{p_1^2d^2 \ell^{4/3}} \\ 
       & \times \Big| \sum_{n_1|p_1d\ell} n_1^{1/3}\sum_{n_2 \leq N_2/n_1^2} \frac{\overline{\lambda_\pi(n_1,n_2)}}{n_2^{1/3}} S(\ell \overline{\gamma}, \mp n_2;p_1d\ell /n_1 ) I_{\pm} \left( \frac{n_1^2n_2}{p_1^3d^3 \ell},y\right)  \Big|^2 \mathrm{d}y.
\end{align}
Next we follow arguments from Section \ref{small n} to see that 
\begin{align}\label{Gp2}
    \sum_{p_1| M_2^{1+b_2}} G_{p_1}^0(1)  \ll    \sum_{p_1| M_2^{1+b_2}} \mathop{\sum_{\substack {e_1e_2 \ll N_1 \\ e_1\mid M_2^{1+b_2} }   }}  \mathop{\sumstar_{ \substack{d \sim \frac{C}{A M_2^{b_2} e_2}
       } } }   \frac{ {N^{4/3}} }{ep_1^2d^2  {\ell^{4/3}}} \mathop{\sum_{ \substack{n_1rs= \ell \\ (n_1s,p_1d)=1 \\ r | (p_1d)^\infty \\ en_1 \asymp N_1  }}}  (en_1)^{2/3}  \int_{-K}^K  \mathcal{L}_3 \,  \mathrm{d}y,
\end{align}
where 
\begin{align*}
    \mathcal{L}_3= \sum_{\substack{ \gamma \shortmod{dp_1} \\ (\gamma, dp_1)=1 }   } \Big|   \sum_{n_2 \leq N_2/(en_1)^2}\frac{\overline{\lambda_\pi(en_1,n_2)}}{ {n_2^{1/3}}  }  S(rsn_1 \overline{\gamma}, \mp n_2;p_1drs) I_{\pm} \left( \frac{e^2n_1^2n_2}{e^3p_1^3d^3 \ell},y\right) \Big|^2.
\end{align*}
We apply Lemma \ref{kloosterman simple} to $ \mathcal{L}_3$ and arrive at 
\begin{align*}
    \mathcal{L}_3 \leq  p_1dr^2s \sum_{\substack{ \gamma \shortmod{dp_1s} \\ (\gamma, dp_1s)=1 }   } \Big| \sum_{rn_2 \leq  N_2/(en_1)^2}  \frac{\overline{\lambda_\pi(en_1,rn_2)}}{ {(rn_2)^{1/3}}  }  e\left( \frac{hn_2}{p_1ds}\right)   I_{\pm} \left( \frac{e^2n_1^2rn_2}{e^3p_1^3d^3 \ell},y\right)  \Big|^2.
\end{align*}
Thus 
\begin{align}\label{l4}
    \sum_{p_1| M_2^{1+b_2}} G_{p_1}^0(1)  \ll    \sum_{p_1| M_2^{1+b_2}} \frac{A M_2^{b_2} }{p_1 C}\mathop{\sum_{\substack {e_1e_2 \ll N_1 \\ e_1\mid M_2^{1+b_2} }   }}     \frac{ {N^{4/3}}  }{  {e_1 \ell^{4/3}}} \mathop{\sum_{ \substack{n_1rs= \ell \\ (n_1s,p_1d)=1 \\ r | (p_1d)^\infty \\ en_1 \asymp N_1  }}} r^2s (en_1)^{2/3}    \mathcal{L}_4,
\end{align}
where 
\begin{align*}
    \mathcal{L}_4 =\int_{-K}^KW(y/K) \mathop{\sumstar_{ \substack{d \sim \frac{C}{A M_2^{b_2} e_2}
       } } }\sum_{\substack{ \gamma \shortmod{dp_1s} \\ (\gamma, dp_1s)=1 }   } \Big| \sum_{rn_2 \leq   N_2/(en_1)^2}  \frac{\overline{\lambda_\pi(en_1,rn_2)}}{ {(rn_2)^{1/3}}  }  e\left( \frac{hn_2}{p_1ds}\right)  I_{\pm} \left( \frac{e^2n_1^2rn_2}{e^3p_1^3d^3 \ell},y\right)   \Big|^2 \mathrm{d}z.
\end{align*}
Here $W$ is an appropriate bump function. 
Recall that 
\begin{align}
    I_{\pm} \left( \frac{e^2n_1^2rn_2}{e^3p_1^3d^3 \ell},y \right)=\int_{0}^{\infty}  e\left(\frac{ z N}{MM_2^{b_2}ae_2d} \right)  \frac{V(z)}{ z^{iy/2+1/3}}  e\left(\pm 3 \frac{(e^2n_1^2rn_2 Nz)^{1/3} }{  ep_1d \ell^{{1/3}}}\right)  \, \mathrm{d} z.
\end{align}
  Opening the absolute valued square in $ \mathcal{L}_4 $
  we get 
  \begin{align*}
        \mathcal{L}_4 = \mathop{\sumstar_{ \substack{d \sim \frac{C}{A M_2^{b_2} e_2}
       } } }\sum_{\substack{ \gamma \shortmod{dp_1s} \\ (\gamma, dp_1s)=1 }   }  \sum_{n_2} \sum_{n_2^\prime} \frac{\overline{\lambda_\pi(en_1,rn_2)}}{ {(rn_2)^{1/3}}  }\frac{{\lambda_\pi(en_1,rn_2^\prime)}}{ {(rn_2^\prime)^{1/3}}  }  e\left( \frac{h(n_2-n_2^\prime)}{p_1ds}\right) \mathcal{J}(\cdots),   \end{align*}
  where 
\begin{align*}
   \mathcal{J}(\cdots)=&\int_{0}^{\infty} W(y/K) \int_{0}^{\infty} {V(z_1)(z_2/z_1)^{1/3} }\frac{V(z_2)}{(z_1/z_2)^{i3y/2}}  e\left(\frac{ (z_1-z_2{}) N}{MM_2^{b_2}ae_2d} \right)  \\
   & \times e\left(\pm 3 \frac{(e^2n_1^2r N)^{1/3}((n_2z_1)^{1/3}-(n_2^\prime z_2)^{1/3}) }{  ep_1d \ell^{{1/3}}}\right)  \, \mathrm{d} z_1 \,  \mathrm{d} z_2 \,  \mathrm{d} y.
\end{align*}
 On applying integration by parts to the $y$-integral, we get that  $|z_1-z_2| \leq M^\epsilon/K$ (remaining contribution being $O(M^{- 2025})$). On writing $z_1=z_2+u,$ $|u| \leq M^{\epsilon}/K,$ we arrive at 
\begin{align*}
   \int_{0}^{\infty} W(y/K) &\int_{0}^{\infty} \frac{V(z_2+u)}{ (1+u/z_2)^{1/3} } \frac{V(z_2)}{(1+u/z_2)^{i3y/2}}  e\left(\frac{ u N}{MM_2^{b_2}ae_2d} \right)   \\
  &\times  e\left(\pm 3 \frac{(e^2n_1^2r N)^{1/3}(n_2^{1/3}(z_2+u)^{1/3}-(n_2^\prime z_2)^{1/3}) }{  ep_1d \ell^{{1/3}}}\right)  \, \mathrm{d} u \, \mathrm{d} z_2  \,\mathrm{d} y.
\end{align*}
Using the binomial expansion 
\[ (1+u/z_2)^{1/3}=1+\frac{u}{3 z_2}+\cdots,  \]
we can write 
\[e\left(\pm 3 \frac{(e^2n_1^2r N)^{1/3}(n_2^{1/3}(z_2+u)^{1/3} }{  ep_1d \ell^{{1/3}}}\right) = e\left(\pm 3 \frac{(e^2n_1^2r N)^{1/3}n_2^{1/3}z_2^{1/3} }{  ep_1d \ell^{{1/3}}}\right)e(g(n_2,z_2,u,d)),  \]
 where $g$ is a flat-function, i.e., $n_2^j\frac{d^j}{dn_2^j}g(n_2,\cdots) \ll M^{\epsilon}$. 
Thus we get  
\[   \mathcal{L}_4 \ll K \int_{|u| \leq \frac{M^\epsilon}{K}}(\mathcal{L}_5 +\mathcal{L}_6  ) \mathrm{d}u,\]
with 
\begin{align}\label{l5}
     \mathcal{L}_5= \int_1^2\mathop{\sumstar_{ \substack{d \sim \frac{C}{A M_2^{b_2} e_2}
       } } }\sum_{\substack{ \gamma \shortmod{dp_1s} \\ (\gamma, dp_1s)=1 }   } | \mathcal{S}_1(\cdots)|^2 \mathrm{d}z_2,
\end{align}
\begin{align*}
        \mathcal{S}_1(\cdots)=\sum_{n_2 \leq  N_2/r(en_1)^2}  \frac{\overline{\lambda_\pi(en_1,rn_2)}}{ {(rn_2)^{1/3}}  }  e\left( \frac{hn_2}{p_1ds}\right) e\left(\pm 3 \frac{(e^2n_1^2r N)^{1/3}n_2^{1/3}z_2^{1/3} }{  ep_1d \ell^{{1/3}}}\right)e(g(\cdots)),
\end{align*}
and
\begin{align*}
     \mathcal{L}_6= \int_1^2\mathop{\sumstar_{ \substack{d \sim \frac{C}{A M_2^{b_2} e_2}
       } } }\sum_{\substack{ \gamma \shortmod{dp_1s} \\ (\gamma, dp_1s)=1 }   } | \mathcal{S}_2(\cdots)|^2 \mathrm{d}z_2,
\end{align*}
\begin{align*}
        \mathcal{S}_2(\cdots)=  \sum_{n_2 \leq  N_2/r(en_1)^2}  \frac{{\lambda_\pi(en_1,rn_2^\prime)}}{ {(rn_2^\prime)^{1/3}}  }  e\left( \frac{-hn_2^\prime}{p_1ds}\right)  e\left(\mp 3 \frac{(e^2n_1^2r N)^{1/3}(n_2^\prime)^{1/3}z_2^{1/3} }{  ep_1d \ell^{{1/3}}}\right).
\end{align*}
\begin{lemma}\label{l4 large sieve}
    We have 
    \[\mathcal{L}_4 \ll \frac{M^\epsilon}{K}  \left(  \frac{p_1N^2s}{M^2 A^2M_2^{2b_2}e_2^2}+  \frac{ e_1  p_1 \ell^{} N^2M_2^{2}} { r(en_1)^{2} M_2^{b_2}A^3 M^3 {}} {  }{ }  \right) \sum_{n_2 \ll   \frac{ N_2}{r (en_1)^2}}     \frac{ {|\lambda_\pi(en_1,rn_2)|^2}}{  {(rn_2)^{2/3}}  }. \]
\end{lemma}
\begin{proof}
We have 
\[   \mathcal{L}_4 \ll K \int_{|u| \leq \frac{M^\epsilon}{K}}(\mathcal{L}_5 +\mathcal{L}_6  ) \, \mathrm{d}u.\]
We will  analyse  $ \mathcal{L}_5$ (see \eqref{l5}), as the  analysis for   $ \mathcal{L}_6$   is similar and $\mathcal{L}_6 \ll  \mathcal{L}_5$. We   get rid of the factor $e(g(\cdots))$ using  Mellin inversion. To this end, we introduce a smooth dyadic partition of unity to the $n_2$-sum. Thus $ | \mathcal{S}_1(\cdots)|^2$ is bounded by 
\begin{align*}
   & M^\epsilon \sup_{N_3 \ll \frac{N_2}{r (en_1)^2}} 
       \Big| \sum_{n_2 \sim N_3} W(\frac{n_2}{N_3})  \frac{\overline{\lambda_\pi(en_1,rn_2)}}{ {(rn_2)^{1/3}}  }  e\left( \frac{hn_2}{p_1ds}\right) e\left( \frac{ \pm 3(e^2n_1^2r N)^{1/3}n_2^{1/3}z_2^{1/3} }{  ep_1d \ell^{{1/3}}}\right) e(g(\cdots))  \Big|^2,
\end{align*}
 On applying Mellin inversion  to 
   $  W(\frac{n_2}{N_3}) e(g(n_2,\cdots)), $    we get that 
\begin{align*}
     | \mathcal{S}_1(\cdots)|^2 &\ll M^\epsilon \sup_{N_3 \ll \frac{e_1^3N_2}{r (en_1)^2}}  \Big| \sum_{n_2 \sim N_3} \frac{\overline{\lambda_\pi(en_1,rn_2)}}{n_2^{i \xi_0}  {(rn_2)^{1/3}}  }  e\left( \frac{hn_2}{p_1ds}\right) e\left(\pm 3 \frac{(e^2n_1^2r N)^{1/3}n_2^{1/3}z_2^{1/3} }{  ep_1d \ell^{{1/3}}}\right)  \Big|^2,
\end{align*}
 for some $\xi_0 \ll M^{\epsilon}$. By   change of variable $z_2^{1/3} \rightarrow z$ followed by $z \rightarrow z/K$,
we see that 
\begin{align*}
     \mathcal{L}_5 \ll \frac{1}{K} \int_K^{2K}\mathop{\sumstar_{ \substack{d \sim \frac{C}{A M_2^{b_2} e_2}
       } } }\sum_{\substack{ \gamma \shortmod{dp_1} \\ (\gamma, dp_1)=1 }   }  \sup_{N_3 \ll \frac{ N_2}{r (en_1)^2}}  \Big| \mathcal{S}_3  \Big|^2 \mathrm{d}z ,
\end{align*}
where 
\[ \mathcal{S}_3= \sum_{n_2 \sim N_3} \frac{\overline{\lambda_\pi(en_1,rn_2)}}{n_2^{i \xi_0}  {(rn_2)^{1/3}}  }  e\left( \frac{hn_2}{p_1ds}\right) e\left(\pm 3 \frac{(e^2n_1^2r N)^{1/3}n_2^{1/3}z^{} }{ K ep_1d \ell^{{1/3}}}\right).\]
On interchanging integration and summations, we get 
\begin{align*}
     \mathcal{L}_5 \ll \frac{1}{K} \sup_{N_3 \ll \frac{N_2}{r (en_1)^2}} \mathop{\sumstar_{ \substack{d \sim \frac{C}{A M_2^{b_2} e_2}
       } } }\sum_{\substack{ \gamma \shortmod{dp_1s} \\ (\gamma, dp_1s)=1 }   }  \int_K^{2K}  \Big| \mathcal{S}_3  \Big|^2 \mathrm{d}z.
\end{align*}
Let $C_1=\frac{C}{AM_2^{b_2}e_2} $. We change the variable $z \rightarrow  \frac{d}{C_1}z$ to see that 
\begin{align*}
     &\mathop{\sumstar_{ \substack{d \sim {C_1}{ }
       } } }\sum_{\substack{ \gamma \shortmod{dp_1s} \\ (\gamma, dp_1s)=1 }   }  \int_K^{2K}  \Big|  \sum_{n_2 \sim N_3} \frac{\overline{\lambda_\pi(en_1,rn_2)}}{n_2^{i \xi_0}  {(rn_2)^{1/3}}  }  e\left( \frac{hn_2}{p_1ds}\right) e\left(\pm 3 \frac{(e^2n_1^2r N)^{1/3}n_2^{1/3}z^{} }{ K ep_1d \ell^{{1/3}}}\right)  \Big|^2 dz \\
       & \ll   \mathop{\sumstar_{ \substack{d \sim  {C_1}{ }
       } } }\sum_{\substack{ \gamma \shortmod{dp_1s} \\ (\gamma, dp_1s)=1 }   }  \int_{\frac{KC_1}{d} }^{ \frac{2KC_1}{d}}  \Big|  \sum_{n_2 \sim N_3} \frac{\overline{\lambda_\pi(en_1,rn_2)}}{n_2^{i \xi_0}  {(rn_2)^{1/3}}  }  e\left( \frac{hn_2}{p_1ds}\right) e\left(\pm 3 \frac{(e^2n_1^2r N)^{1/3}n_2^{1/3}z^{} }{ K ep_1C_1 \ell^{{1/3}}}\right)  \Big|^2 dz  \\
        & \ll   \mathop{\sumstar_{ \substack{d \sim  {C_1}{ }
       } } }\sum_{\substack{ \gamma \shortmod{dp_1s} \\ (\gamma, dp_1s)=1 }   }  \int_{{K/2}{} }^{ {2K}{}}  \Big|  \sum_{n_2 \sim N_3} \frac{\overline{\lambda_\pi(en_1,rn_2)}}{n_2^{i \xi_0}  {(rn_2)^{1/3}}  }  e\left( \frac{hn_2}{p_1ds}\right) e\left(\pm 3 \frac{(e^2n_1^2r N)^{1/3}n_2^{1/3}z^{} }{ K ep_1C_1 \ell^{{1/3}}}\right)  \Big|^2 dz  \\
       &\ll M^\epsilon ( C_1^2p_1sK+ \frac{ N_3^{2/3} K ep_1C_1 \ell^{1/3} } {(e^2n_1^2rN)^{1/3}}) \sum_{n_2 \sim N_3} \frac{ {|\lambda_\pi(en_1,rn_2)|^2}}{  {(rn_2)^{2/3}}  },
\end{align*}
where in the last inequality we applied Lemma \ref{hybrid large-sieve with level}. We plug in  the bounds $$C \ll M^{\epsilon}\frac{N}{M}, \ p_1 \leq M_2^{1+b_2}, \ K \ll M^\epsilon \frac{N}{CM},\  N_3 \ll \frac{ N_2}{r(en_1)^2} \ll M^\epsilon\frac{ M_2^3 N^3\ell  }{r(en_1)^2 M^3  A^3N},$$ to get the lemma. 
Note that 
\begin{align*}
   & C_1^2p_1sK+ \frac{ N_3^{2/3} K ep_1C_1 \ell^{1/3} } {(e^2n_1^2rN)^{1/3}}  \\
   &\ll \frac{p_1N Cs}{M A^2M_2^{2b_2}e_2^2}+  \frac{  K ep_1C_1 \ell^{1/3} } {(e^2n_1^2rN)^{1/3}} \frac{ N_2^{2/3}}{r^{2/3}(en_1)^{4/3}} \\
    & \ll M^\epsilon\frac{p_1N^2}{M^2 A^2M_2^{2b_2}e_2^2}+  \frac{   ep_1 \ell^{1/3} N} {M AM_2^{b_2}e_2(e^2n_1^2rN)^{1/3}} \frac{\ell^{2/3}M_2^2N^{4/3} }{r^{2/3}(en_1)^{4/3} M^2A^2} \\
      & \ll M^\epsilon \frac{p_1N^2s}{M^2 A^2M_2^{2b_2}e_2^2}+  \frac{   p_1 \ell^{} N^2} {M^3 {}} \frac{e_1 M_2^2 }{r(en_1)^{2} M_2^{b_2}A^3} \\
       & \ll M^\epsilon \frac{p_1N^2s}{M^2 A^2M_2^{2b_2}e_2^2}+  \frac{   p_1 \ell^{} N^2M_2^{2}} {M^3 {}} \frac{e_1  }{r(en_1)^{2} M_2^{b_2}A^3} .
\end{align*}
Thus 
\begin{align*}
     \mathcal{L}_5 \ll \frac{M^\epsilon}{K}  \left(  \frac{p_1N^2s}{M^2 A^2M_2^{2b_2}e_2^2}+  \frac{   p_1 \ell^{} N^2M_2^{2}} {M^3 {}} \frac{e_1  }{r(en_1)^{2} M_2^{b_2}A^3}  \right) \sum_{n_2 \ll   \frac{N_2}{r (en_1)^2}}     \frac{ {|\lambda_\pi(en_1,rn_2)|^2}}{  {(rn_2)^{2/3}}  }.
\end{align*}
\end{proof}  

\begin{lemma}\label{large n bound}
For $ n_1^2n_2N/(p_1d)^3 \ell > M^\epsilon$,
    we have $C \ll M^{-\epsilon}N/M$ and 
    \begin{align*}
        \mathop{\sum}_{\substack{p_1 \mid M_2^{1+b_2} }}  G_{p_1}^{0}(\ell, A,C,N) \ll \frac{N^2  M_2}{A^2 \sqrt{M}}.
    \end{align*}
\end{lemma}
\begin{proof}
\eqref{N0} and   $n_1^2n_2N/(p_1d)^3 \ell > M^\epsilon $ implies $C \ll M^{-\epsilon}N/M $. 
    On plugging in   Lemma \eqref{l4 large sieve} into \eqref{l4} we see that  $ \sum_{p_1| M_2^{1+b_2}} G_{p_1}^0(1)$ is 
\begin{align}\label{2 terms}
     &\ll    \sum_{p_1| M_2^{1+b_2}} \frac{A M_2^{b_2} }{p_1 C}\mathop{\sum_{\substack {e_1e_2 \ll N_1 \\ e_1\mid M_2^{1+b_2} }   }}     \frac{ {N^{4/3}}  }{  {e_1 \ell^{4/3}}} \mathop{\sum_{ \substack{n_1rs= \ell \\ (n_1s,p_1d)=1 \\ r | (p_1d)^\infty \\ en_1 \asymp N_1  }}} r^2s (en_1)^{2/3}  \\
    & \times \frac{M^\epsilon}{K}  \left(  \frac{p_1N^2s}{M^2 A^2M_2^{2b_2}e_2^2}+  \frac{ e_1  p_1 \ell^{} N^2M_2^{2}} { r(e_2n_1)^{2} M_2^{b_2}A^3 M^3 {}} {  }{ }  \right) \sum_{n_2 \ll   \frac{N_2}{r (en_1)^2}}     \frac{ {|\lambda_\pi(en_1,rn_2)|^2}}{  {(rn_2)^{2/3}}  }.
\end{align}
Consider  the first term 
\begin{align}\label{first}
         &\sum_{p_1| M_2^{1+b_2}} \frac{A M_2^{b_2} }{p_1 C K}\mathop{\sum_{\substack {e_1e_2 \ll N_1 \\ e_1\mid M_2^{1+b_2} }   }}     \frac{ {N^{4/3}}  }{  {e_1 \ell^{4/3}}} \mathop{\sum_{ \substack{n_1rs= \ell \\ (n_1s,p_1d)=1 \\ r | (p_1d)^\infty \\ en_1 \asymp N_1  }}}  
         \frac{r^2s^2 (en_1)^{2/3}  p_1N^2}{M^2 A^2M_2^{2b_2}e_2^2}  \sum_{n_2 \ll   \frac{N_2}{r (en_1)^2}}     \frac{ {|\lambda_\pi(en_1,rn_2)|^2}}{  {(rn_2)^{2/3}}  } \notag \\
         =&  \sum_{p_1| M_2^{1+b_2}} \frac{N^{4/3+2}  }{M^2 AM_2^{b_2} C K}\mathop{\sum_{\substack {e_1e_2 \ll N_1 \\ e_1\mid M_2^{1+b_2} }   }}  \frac{\ell^{2/3}}{e_1e_2^2 }  \mathop{\sum_{ \substack{e|n_1 \asymp N_1   }}}     \frac{  e^{2/3}}{ n_1^{4/3} }   \sum_{rn_2 \leq N_2/(en_1)^2}\frac{ { |\lambda_\pi(en_1,rn_2)|^2 }}{ { (rn_2)^{2/3}}  }  
\end{align}
We change the variable $rn_2 \rightarrow n_2$, $en_1 \rightarrow n_1$,  to arrive at 
\begin{align}
       & \sum_{p_1| M_2^{1+b_2}} \frac{N^{4/3+2}  }{M^2 AM_2^{b_2} C K}\mathop{\sum_{\substack {e_1e_2 \ll N_1 \\ e_1\mid M_2^{1+b_2} }   }}  \frac{\ell^{2/3}}{e_1e_2^2 }  \mathop{\sum_{ \substack{n_1rs= \ell \\ (n_1s,p_1d)=1 \\ r | (p_1d)^\infty \\ e|n_1 \asymp N_1  }}}     \frac{  e^{2}}{ n_1^{4/3} }   \sum_{n_2 \leq N_2/n_1^2} n_2^{1/3}\frac{ { |\lambda_\pi(n_1,n_2)|^2 }}{ { n_2^{}}  }  \\
        & \ll  \sum_{p_1| M_2^{1+b_2}} \frac{N^{4/3+2}  }{M^2 AM_2^{b_2} C K}\mathop{\sum_{\substack {e_1e_2 \ll N_1 \\ e_1\mid M_2^{1+b_2} }   }}  \frac{\ell^{2/3}}{e_1e_2^2 }  \mathop{\sum_{ \substack{n_1rs= \ell \\ (n_1s,p_1d)=1 \\ r | (p_1d)^\infty \\ e|n_1 \asymp N_1  }}}    \frac{  e^{2}N_2^{1/3} }{ e^{} }   \sum_{n_2 \leq N_2/n_1^2}\frac{ { |\lambda_\pi(n_1,n_2)|^2 }}{ { n_1n_2^{}}  }  \\
          & \ll  \sum_{p_1| M_2^{1+b_2}} \frac{N^{4/3+2}  }{M^2 AM_2^{b_2} C K}\mathop{\sum_{\substack {e_1e_2 \ll N_1 \\ e_1\mid M_2^{1+b_2} }   }}  \frac{\ell^{2/3}}{e_1e_2^2 }     {  e^{}N_2^{1/3} }{  }  \\
           & \ll \frac{N^{4/3+2}\ell^{2/3}N_2^{1/3}  }{M^2 AM_2^{b_2} C K} \ll \frac{N^{4/3+2}\ell^{2/3} }{M^2 AM_2^{b_2} C K} \times \frac{M_2 N \ell^{1/3}}{MAN^{1/3}} \ll  M^\epsilon \frac{N^2  M_2}{M_2^{b_2} A^2 \sqrt{M}}.
\end{align}
Next we consider the second term in \eqref{2 terms}.
Note that 
\begin{align*}
      \frac{ e_1  p_1 \ell^{} N^2M_2^{2}} { r(en_1)^{2} M_2^{b_2}A^3 M^3 {}} {  }{ } =  \frac{p_1N^2s}{M^2 A^2M_2^{2b_2}e_2^2} \times \frac{M_2^{2+b_2} \ell }{MAe_1 rs n_1^2  },
\end{align*}
and (as $\ell=rsn_1$)
\begin{align*}
      \frac{M_2^{2+b_2} \ell }{MAe_1 rs n_1^2  } \ll       \frac{M_2^{2+b_2}  }{M   }.
\end{align*}
Thus the contribution of the second term in \eqref{2 terms} to 
 $ \sum_{p_1| M_2^{1+b_2}} G_{p_1}^0(1)$ is 
 \[\ll  M^\epsilon \frac{N^2  M_2}{M_2^{b_2} A^2 \sqrt{M}} \times \frac{M_2^{2+b_2}  }{M   }\ll M^\epsilon \frac{N^2  M_2^3}{ A^2 {M^{3/2}}}  \ll  M^\epsilon \frac{N^2  M_2}{ A^2 {M^{1/2}}}.  \]
 The lemma follows. 
 \end{proof}
\section{Conclusion}
We plug in  estimates from Lemma \ref{easy bound}, \ref{small n bound}, \ref{large n bound} into \eqref{OD after a} to get 
\begin{align}
     \mathrm{OD}_{\pi}^{(2, +)}(\cdots)  &\ll  \sup_{A \ll C}  \sum_{a \sim A} \sum_{b_1,b_2 \ll \log C } G_{}^{b_1,b_2}(\ell, A,C,N) \\
     & \ll M^{\epsilon} C^2\left(N+ \frac{N^2}{M_1^2}\right)+ {NK^2 {\ell^2} C^2 M_2}{    }  \left(1 +\frac{C M_2^2}{ N}\right)+\frac{N^2  M_2}{ \sqrt{M}}. 
\end{align}
Same estimates hold for $ \mathrm{OD}_{\pi}^{(2, -)}(\cdots) $ and $ \mathrm{OD}_{\pi}^{(1, \pm)}(\cdots) $. Thus from \eqref{od after am-gm} we have 
\begin{align*}
    \mathrm{OD}_{\pi}(\ell,C,  N) &\ll  M^{\epsilon}\min\left\{\frac{M_1}{{N}}, \frac{M_1N^2}{C^3M^3}\right\} \left(  C^2\left(N+ \frac{N^2}{M_1^2}\right)+ {NK^2 {\ell^2} C^2 M_2}{    }  \left(1 +\frac{C M_2^2}{ N}\right)+\frac{N^2  M_2}{ \sqrt{M}}\right)  
\end{align*}
We estimate the first term as follows:
\begin{align*}
     &  M^{\epsilon}\min\left\{\frac{M_1}{{N}}, \frac{M_1N^2}{C^3M^3}\right\} C^2(N+N^2/M_1^2) \\
     \ll& M^{\epsilon} \min\left\{\frac{M_1}{{N}}, \frac{M_1N}{C^2M^2}\right\} C^2(N+N^2/M_1^2) \\
    \ll&  M^{\epsilon}\frac{M_1N}{C^2M^2} C^2(N+N^2/M_1^2)   \ll M^{\epsilon}M_1N.
\end{align*}
The second term is 
\begin{align}\label{2nd}
    M^{\epsilon}\min\left\{\frac{M_1}{{N}}, \frac{M_1N^2}{C^3M^3}\right\} {NK^2 {\ell^2} C^2 M_2}{    }  \left(1 +\frac{C M_2^2}{ N}\right).
\end{align}
For $C \ll M^{-\epsilon}N/M$, \eqref{2nd} is  bounded by (using $N\ell^2 \ll M^{3/2+\epsilon}$)
$$   M^\epsilon \frac{M_1}{N}N \frac{N^2}{C^2M^2}\ell^2 C^2M_2 \ll M^\epsilon M_1N. $$
For $C \gg M^{-\epsilon}N/M$, we have $K \ll M^\epsilon$. Thus  \eqref{2nd} is  bounded by 
\[ M^\epsilon \frac{M_1N^2}{C^3M^3} N \ell^2C^2M_2+ M^\epsilon \frac{M_1N^2}{C^3M^3}  \ell^2C^3M_2^3  \ll M^{\epsilon}M_1N.  \]
Similarly the third term  
\[ M^{\epsilon}\min\left\{\frac{M_1}{{N}}, \frac{M_1N^2}{C^3M^3}\right\} \frac{N^2M_2}{\sqrt{M}} \ll M^{\epsilon} \frac{M_1}{{N}} \frac{N^2M_2}{\sqrt{M}} \ll M^{\epsilon}M_1N.  \]

On plugging in the above bound in \eqref{S M after peter}, we get 
\begin{align}
    \mathcal{S}_\pi(M) 
     \ll \frac{M_1-1}{2}+\sup_\ell \sup_{N\ell^2 \ll M^{3/2+\epsilon}} \sup_{ C}\frac{|\mathrm{OD}_{\pi}(\ell,C, N)| }{N} \ll M^\epsilon M_1. 
\end{align}
This proves   Theorem \ref{main thm}.

{} 

\end{document}